\documentclass[11pt,a4paper]{article}
\usepackage[a4paper,margin=30mm]{geometry}
\usepackage[T1]{fontenc}
\usepackage[utf8]{inputenc}
\usepackage{mathpazo,courier,microtype}
\usepackage{amsmath,amssymb,amsthm,mathtools,mathrsfs}
\usepackage[numbers, sort&compress]{natbib}
\usepackage{enumitem,xcolor,hyperref}
\hypersetup{
	hypertexnames=false,
	colorlinks,
	citecolor=blue,
	filecolor=blue,
	linkcolor=blue,
	urlcolor=black
}
\numberwithin{equation}{section}
\newtheorem{theorem}{Theorem}[section]
\newtheorem{proposition}[theorem]{Proposition}
\newtheorem{lemma}[theorem]{Lemma}
\newtheorem{corollary}[theorem]{Corollary}

\newtheorem{remark}[theorem]{Remark}

\newcommand{\R}{\mathbb R}
\newcommand{\N}{\mathbb N}
\newcommand{\Q}{\mathbb Q}
\newcommand{\Om}{\Omega}
\newcommand{\bOm}{\overline\Omega}
\newcommand{\dd}{\,\mathrm d}
\newcommand{\Br}{\mathcal B_r}
\newcommand{\Prob}{\mathcal P}
\newcommand{\supp}{\operatorname{supp}}
\newcommand{\dist}{\operatorname{dist}}
\newcommand{\Lip}{\operatorname{Lip}}
\newcommand{\card}{\operatorname{card}}

\newcommand{\Id}{\operatorname{Id}}

\newcommand{\degLS}{\operatorname{deg}_{\mathrm{LS},E}}
\newcommand{\degB}{\operatorname{deg}_{\mathrm B}}
\newcommand{\weak}{\rightharpoonup}
\newcommand{\norm}[1]{\lVert#1\rVert}
\newcommand{\pair}[2]{\langle#1,#2\rangle}
\newcommand{\Fcal}{\mathcal F}
\newcommand{\Gcal}{\mathcal G}
\newcommand{\Mcap}{\mathcal M}

\title{\bfseries Topological degree for singular $N$-mean field equations}
\author{Mohameden Ahmedou\thanks{Supported by the Deutsche Forschungsgemeinschaft (DFG) under grant AH 156/2-1.},\and Zhongwei Tang\thanks{Supported by NSFC grants 12471099 and 12271223.},\and Heming Wang\thanks{Supported by China Postdoctoral Science Foundation 2026M793418.}}
\date{}
\begin{document}
\maketitle
\begin{abstract}
We determine the Leray--Schauder degree of the singular Dirichlet $N$-mean field equation on a smooth bounded domain in $\R^N$, $N\geq2$, for every positive nonresonant parameter and singular exponents $-1<\alpha_i\leq0$. The formula depends only on the Euler characteristic of the domain and the weights $(1+\alpha_i)^{N-1}$. We also identify the homotopy type of every sufficiently negative variational sublevel with the corresponding weighted barycenter space. 

We prove quantization of interior concentration masses and exclude boundary concentration and residual diffuse mass. The degree computation uses common descent fields for nearby parameters, bounds for their trajectories, and a finite-dimensional comparison with the Euler characteristic of negative sublevels. Annular averaging and dilation provide local contractions, which are extended to establish the sublevel homotopy equivalence. The arguments apply to the nonlinear, possibly degenerate $N$-Laplacian with singular weights, without requiring a global Palais--Smale condition or isolation and nondegeneracy of solutions.
\end{abstract}

\tableofcontents

\section{Introduction}
\subsection{Background and motivation}
Let $\Om\subset\R^N$, $N\geq2$, be a smooth bounded connected domain, and let $P=\{p_1,\ldots,p_m\}\subset\Om$ consist of distinct points. We consider the quasilinear mean field equation
\begin{equation}\label{maineq}
\left\{\begin{aligned}
   & -\Delta_Nu=\rho\frac{H(x)e^{u(x)}}{\int_\Om H(y)e^{u(y)}\dd y},&&x\in\Om,\\ &u=0,&&x\in\partial\Om,
\end{aligned}\right.
\end{equation}
where $\Delta_Nu=\operatorname{div}(|\nabla u|^{N-2}\nabla u)$, $\rho\geq0$ is a parameter, and $H(x)=h(x)\prod_{i=1}^m|x-p_i|^{N\alpha_i}$ with
\begin{equation}\label{eq:assumptions}
h\in C^1(\bOm),\quad \min_{x\in\bOm}h(x)>0,\quad -1<\alpha_i\leq0\quad(1\leq i\leq m).
\end{equation}
We allow $m=0$, interpret empty products as one, and set each factor
with $\alpha_i=0$ equal to one also at $p_i$. No restriction is imposed
on the topology of $\Om$. We refer to \eqref{maineq} as the singular $N$-mean field equation.

The mean field normalization makes the integral of the right-hand side of \eqref{maineq} equal to $\rho$. When $P=\varnothing$ or all exponents vanish, \eqref{maineq} is the regular $N$-mean field equation. In dimension $N=2$, mean field equations on planar domains and closed Riemannian surfaces have attracted considerable interest through their connections with conformal geometry (see, e.g., \cite{CGY93,CY87,KW74}), statistical mechanics (see, e.g., \cite{CLMP92,CLMP95,CK94,Kiessling93}), turbulent Euler flows (see, e.g., \cite{DPEM10,Stuart67}), self-dual condensates in Chern--Simons--Higgs models (see, e.g., \cite{CLW04,DPEFM15,DJLW98,EF14,LW10,LY13,NT99}), and stationary Keller--Segel systems (see, e.g., \cite{ABH26,AHW2025-1,Battaglia19}). The exponent $N$ is distinguished by the invariance of the Dirichlet energy under spatial dilations and the exponential integrability of $W^{1,N}_0(\Om)$. Equation \eqref{maineq} is therefore a natural quasilinear counterpart of the two-dimensional problem, in which concentration competes with the topology of the domain.

Set
\begin{equation}\label{eq:functional}
\begin{gathered}
E=W^{1,N}_0(\Om),\quad \norm{u}_E=\norm{\nabla u}_{L^N(\Om)},\quad Q(u)=\frac1N\int_\Om|\nabla u|^N\dd x,\\
M(u)=\int_\Om He^u\dd x,\quad B(u)=\ln  M(u),\quad J_{\rho,H}(u)=Q(u)-\rho B(u).
\end{gathered}
\end{equation}
The restriction $\alpha_i>-1$ implies $H\in L^s(\Om)$ for some $s>1$. The exponential term is consequently continuously Fr\'echet differentiable on $E$, although $H$ can be unbounded at the points of $P$.

In two dimensions, Li \cite{L1999} showed that, at each fixed nonresonant parameter, the degree of the regular surface equation is determined by the Euler characteristic and proposed computing it through its jumps at blow-up parameters. Subsequently, Chen and Lin established sharp multi-bubble estimates in \cite{CL2002} and computed explicit degree formulas for the regular equation on closed surfaces and planar Dirichlet domains in \cite{CL03}; they also extended the formulas to singular mean field equations with arbitrary exponents $\alpha_i>-1$ in  \cite{CL15}. Related degree formulas for two-dimensional mean field type equations can be found, for instance, in \cite{JWY18,AHW2025-1,ABL2017,Malchiodi26,AB2017,Malchiodi}.  In particular, Ahmedou, Ben Ayed, and Lucia
\cite{ABL2017} and Ahmedou and Ben Ayed \cite{AB2017} obtained
Leray--Schauder degree formulas at resonant parameters for the regular
Dirichlet mean field equation and its singular counterpart with positive
nonintegral singular exponents, respectively, via Bahri's theory of critical points at infinity \cite{Bahri}. For a recent survey, see \cite{Malchiodi26} and the references therein.

From a variational viewpoint, Malchiodi \cite{Malchiodi}  gave a Morse-theoretic interpretation of the regular surface degree through variational sublevels. Carlotto and Malchiodi \cite{CM}  introduced weighted barycenter spaces for singular Liouville equations on closed surfaces and related them to negative sublevels through an improved concentration inequality and concentrating test functions. Their construction includes an interpolation between regular and singular profiles as a center approaches a point of $P$. Kallel \cite{Kallel} computed the Euler characteristic of weighted barycenter spaces and its relation to the Chen--Lin generating series \cite{CL03,CL15}.

For $N\geq2$, Esposito and Morlando \cite{EM} established
mass quantization and nonresonant compactness for the Dirichlet
$N$-mean field equation with smooth positive weights, and proved
existence in nonresonant intervals provided that the corresponding
spaces of formal barycenters are noncontractible.

The regular theory raises a natural question: how do interior power singularities change the concentration thresholds and the topology responsible for existence? Our purpose is to establish the degree and the full negative-sublevel homotopy identification for the singular Dirichlet problem \eqref{maineq} in arbitrary dimension. The concentration mass at $p_i$ is proportional to $(1+\alpha_i)^{N-1}$, rather than $1+\alpha_i$ except when $N=2$. Our approach to computing the degree therefore requires both the degree--Euler comparison and the sublevel homotopy identification for the singular quasilinear functional.

Several obstacles distinguish this problem from its semilinear counterpart. For $N>2$, the energy $Q$ is not quadratic, the principal operator degenerates at zero gradients, and its inverse is nonlinear. Thus Green-function superposition and a linear compact-operator gradient representation are unavailable in this case. The range of negative exponents in \eqref{eq:assumptions} does not in general allow global $C^1$ regularity at the points of $P$. We work in $W^{1,N}_0(\Om)$, use the compact nonlinear solution map for degree, and use the monotonicity of the $N$-Laplacian for convergence and finite-dimensional comparison. Concentration asymptotics are obtained from a single nonlinear measure-data limit, rather than a sum of fundamental solutions.

\subsection{Main results}\label{subsec:main}
Throughout the paper, unless otherwise stated, $N\geq2$,
$\Om\subset\R^N$ is a smooth bounded connected domain, and
$P=\{p_1,\ldots,p_m\}\subset\Om$ consists of distinct points.
The weight $H(x)=h(x)\prod_{i=1}^m|x-p_i|^{N\alpha_i}$ satisfies \eqref{eq:assumptions}.

Write $\N=\{1,2,\ldots\}$ and $\N_0=\N\cup\{0\}$, and set
\begin{equation}\label{eq:constants}
q=\frac N{N-1},\quad \gamma_N=\frac{N^2}{N-1}=Nq,\quad S_N=|\mathbb S^{N-1}|,\quad c_N=S_N\gamma_N^{N-1},\quad \beta_i=(1+\alpha_i)^{N-1}.
\end{equation}
In particular $c_2=8\pi$. For $I\subset\{1,\ldots,m\}$ set $\beta_I=\sum_{i\in I}\beta_i$, with $\beta_\varnothing=0$, and define
\begin{equation}\label{eq:resonance}
\Sigma=c_N\{\ell+\beta_I:\ell\in\N_0,\ I\subset\{1,\ldots,m\},\ \ell+|I|>0\}.
\end{equation}
The set $\Sigma$ in \eqref{eq:resonance} is locally finite and does not contain zero. We call $\rho\geq0$ nonresonant when $\rho\notin\Sigma$. For the normalized parameter $r=\rho/c_N$, nonresonance means $c_Nr\notin\Sigma$. The set lists possible aggregate masses; it does not assert that every listed mass is realized by blow-up.

Let $E^*$ be the continuous dual of $E$, with pairing $\pair{\cdot}{\cdot}$. The weak Dirichlet $N$-Laplacian $A:E\to E^*$ is defined by $\pair{A(u)}v=\int_\Om|\nabla u|^{N-2}\nabla u\cdot\nabla v\dd x$ for $u,v\in E$. We set
\begin{equation}\label{eq:solutionmap}
K_{\rho,H}(u)=A^{-1}\Big(\rho\frac{He^u}{M(u)}\Big).
\end{equation}
Proposition~\ref{prop:functional} shows that $K_{\rho,H}$ is compact on bounded subsets of $E$, while Theorem~\ref{thm:compact} bounds its fixed points at nonresonant parameters. Consequently, the degree
\begin{equation}\label{eq:degreedef}
d_\rho(\Om,H)=\degLS(\Id-K_{\rho,H},B_R^E,0)
\end{equation}
is independent of sufficiently large $R$, where $B_R^E=\{u\in E:\norm{u}_E<R\}$. Here $\chi(\Om)$ denotes the ordinary Euler characteristic. For $z\in\mathbb Z$, $k\in\N_0$, we use $\binom zk=z(z-1)\cdots(z-k+1)/k!$, with $\binom z0=1$.

We are now in a position to state our main result, which gives
an explicit formula for the Leray--Schauder degree of
\eqref{maineq} at nonresonant parameters.
\begin{theorem}\label{thm:degree}
Assume \eqref{eq:assumptions}. If $\rho>0$ is nonresonant, set $r=\rho/c_N$ and $k_I=\lfloor r-\beta_I\rfloor$ whenever $\beta_I<r$. Then
\begin{equation}\label{eq:degree}
d_\rho(\Om,H)=\sum_{\substack{I\subset\{1,\ldots,m\}\\\beta_I<r}}(-1)^{|I|}\binom{k_I-\chi(\Om)+m}{k_I}.
\end{equation}
Also $d_0(\Om,H)=1$. The degree is independent of the positive factor $h$ and of the positions of the distinct interior points in $P$. No isolation or nondegeneracy assumption is made on the solution set.
\end{theorem}
\begin{remark}
    When $m=0$, \eqref{eq:degree} is $\binom{k-\chi(\Om)}k$ with $k=\lfloor\rho/c_N\rfloor$. For $N=2$ and a domain with $g$ holes, $\chi(\Om)=1-g$ and
\[
d_\rho=\sum_{\substack{I\subset\{1,\ldots,m\}\\\beta_I<\rho/(8\pi)}}(-1)^{|I|}
\binom{\lfloor\rho/(8\pi)-\beta_I\rfloor+g+m-1}{\lfloor\rho/(8\pi)-\beta_I\rfloor},\quad  \beta_i=1+\alpha_i.
\]
For $N=2$, \eqref{eq:degree} recovers the singular Dirichlet degree formula of Chen--Lin \cite[Theorem 1.4]{CL15}. To compare the formulas, distinguish the generating series from the sum of its coefficients below the parameter. In dimension two set
\[
G_\Om(t)=(1-t)^{\chi(\Om)-m}\prod_{i=1}^m(1-t^{1+\alpha_i})=\sum_{\lambda\geq0}a_\lambda t^\lambda.
\]
After equal exponents are combined, $d_\rho=\sum_{0\leq\lambda<\rho/(8\pi)}a_\lambda$. The same coefficient-sum convention is used in \cite[Theorems 1.1 and 1.4]{CL15}; see also \cite[Corollary 1.8]{Kallel}. No extra factor $(1-t)^{-1}$ precedes the sum $\sum_{0\leq\lambda<\rho/(8\pi)}a_\lambda$. In the common range $-1<\alpha_i\leq0$ the algebraic dependence on the weights and Euler characteristic agrees.
\end{remark}
The existence property of the Leray--Schauder degree gives the following consequence.
\begin{corollary}\label{cor:existence}
Under the assumptions of Theorem~\ref{thm:degree}, if $d_\rho(\Om,H)\ne0$, then \eqref{maineq} has a weak solution $u\in E\cap C^0(\bOm)$. On every compact subset of $\bOm\setminus P$, the solution belongs to $C^{1,\eta}$ for some $\eta>0$ depending on that subset.
\end{corollary}

The regularity in Corollary~\ref{cor:existence} follows from Proposition~\ref{prop:functional}. For $N=2$, local Schauder estimates give $u\in C^2(\Om\setminus P)$, so the equation holds classically there; if all exponents vanish, this holds throughout $\Om$. A global $C^2$ assertion is false in the general setting. Even with $H\equiv1$ and $N>2$, a $C^2$ solution at $\rho>0$ would attain a positive maximum in the interior, where $\nabla u=0$ and $\Delta_Nu=0$, contradicting the strictly positive right-hand side of \eqref{maineq}. Moreover, for $\rho>0$, a solution cannot be $C^1$ near a point $p_i$ with $\alpha_i\leq-1/N$, as proved in \eqref{eq:C1obstruction} below.

A point with $\beta_i=1$ can be deleted without changing the equation, resonance set, or formula; the algebraic cancellation is proved in Proposition~\ref{prop:zero}. Thus the regular case is included without a limiting argument.

To describe the topology underlying \eqref{eq:degree}, choose a compact finite Euclidean polyhedral core $K\Subset\Om$, containing $P$, and a collar deformation retraction of $\bOm$ onto $K$ fixing a neighborhood of $P$. The construction is given in Section~\ref{subsec:barycenters}. For compact $P\subset Y\subset\bOm$, let $\Prob(Y)$ be the probability measures on $Y$ with the weak topology, and let $\delta_y$ denote the unit Dirac measure at $y$. Give ordinary support points weight one and $p_i$ weight $\beta_i$. Thus, for a finitely supported probability measure $\sigma$,
\[
\card_\beta(\sigma)=\#(\supp\sigma\setminus P)+\sum_{\substack{1\leq i\leq m\\p_i\in\supp\sigma}}\beta_i.
\]
The associated weighted barycenter space is
\begin{equation}\label{eq:barycenters}
\Br(Y)=\{\sigma\in\Prob(Y):\supp\sigma\text{ is finite},\ \card_\beta(\sigma)<r\}.
\end{equation}
Only support points with positive coefficients are counted, and $\simeq$ denotes homotopy equivalence. When $H$ and $\Om$ are fixed, write $J_\rho=J_{\rho,H}$ and $d_\rho=d_\rho(\Om,H)$.

\begin{theorem}\label{thm:homotopy}
Assume \eqref{eq:assumptions}. If $\rho>0$ is nonresonant, then for all sufficiently large $L$,
\begin{equation}\label{eq:mainhomotopy}
J_{\rho,H}^{-L}:=\{u\in E:J_{\rho,H}(u)\leq-L\}\simeq\Br(K),\quad 
d_\rho=1-\chi(J_{\rho,H}^{-L})=1-\chi(\Br(K)).
\end{equation}
If $r<\min\{1,\beta_1,\ldots,\beta_m\}$, both spaces are empty and have Euler characteristic zero.
\end{theorem}

The vanishing of \eqref{eq:degree} does not imply nonexistence. A nonempty, noncontractible $\Br(K)$ also implies the existence of a solution: if $J_\rho$ had no critical point, Propositions~\ref{prop:strip} and \ref{prop:high} of Section~\ref{sec:deformation} would identify a sufficiently negative sublevel with a contractible high sublevel, contrary to \eqref{eq:mainhomotopy}.

Formula \eqref{eq:degree} simplifies when there is only one point in $P$, or when all weights are equal. 
\begin{corollary}\label{cor:onemark}
Assume \eqref{eq:assumptions}, and suppose $m=1$ and $\beta_1=\beta\in(0,1]$. Write $r=k+\varepsilon$, $k\in\N_0$, $\varepsilon\in (0,1)$. At a nonresonant parameter,
\[
d_\rho=\left\{\begin{aligned}
    &\binom{k-\chi(\Om)}k&&\text{ if }\, \varepsilon\in (0,\beta),\\&0&&\text{ if }\,\varepsilon\in (\beta,1).
\end{aligned}\right.
\]
\end{corollary}

\begin{corollary}\label{cor:equalmarks}
Assume \eqref{eq:assumptions}, and suppose all $m\geq1$ points of $P$ have weight $\beta$. For $r\in (0,1)$ with $c_Nr\notin\Sigma$, set $j=\lfloor r/\beta\rfloor$. If $j<m$, $d_\rho=(-1)^j\binom{m-1}j$; if $j\geq m$, $d_\rho=0$. In particular, if $m=2$ and $\beta<r<\min\{2\beta,1\}$, the degree is $-1$ on every domain under consideration, and \eqref{maineq} has a solution.
\end{corollary}

\subsection{Proof strategy and scope}
The proofs separate concentration, degree comparison, and the topology of negative sublevels. Remarks~\ref{rem:sharp} and \ref{rem:positive} discuss the parameter restrictions and the obstructions for positive exponents.
The first stage establishes compactness on compact parameter sets avoiding $\Sigma$. A local exponential estimate locates finitely many concentration points. At a possible boundary concentration point, uniqueness for measure-data solutions below the energy class identifies the limit with a regular boundary solution; a boundary-cap Pohozaev identity then excludes concentration at that point. For an interior point, positivity and the logarithmic singularity of the nonlinear limit force the normalization factor to vanish, eliminating diffuse mass. The interior Pohozaev identity fixes the total concentration mass at $c_N$ or $c_N(1+\alpha_i)^{N-1}$. Because the mass is measured before any rescaling, this calculation includes mass carried by collapsed peaks and intermediate scales. The normalization agrees with \cite[Theorem 1.4]{Esposito}; the entire weighted solutions in the nonpositive-exponent range are classified in \cite[Theorems 1.1--1.2]{CEL}. These entire-space results provide context for the mass constant; the boundary exclusion and aggregate quantization for \eqref{maineq} are proved in Section~\ref{sec:compact}.

Chen--Lin's degree computations use sharp blow-up estimates \cite{CL2002,CL03,CL15} to determine the local contributions of concentrated solution families and hence the jumps across resonances. The reduction involves linear Green functions and their regular parts, together with spectral estimates after the bubble parameters have been separated \cite[Sections 2--4]{CL15}. For $N>2$, the Dirichlet operator $A=-\Delta_N$ in \eqref{eq:solutionmap} satisfies $A^{-1}(tf)=t^{1/(N-1)}A^{-1}(f)$ for $t>0$; its inverse is nonlinear, and the field generated by several concentration points cannot be obtained by Green-function superposition. Its linearized principal operator also degenerates wherever $\nabla u=0$. A corresponding local degree computation would therefore require new interaction expansions and uniform estimates for the linearized remainder, beyond the aggregate mass quantization established here. This does not rule out such a method; it explains why the semilinear reduction cannot simply be transferred. Our degree--Euler comparison avoids these additional expansions and estimates, while Section~\ref{sec:topology} establishes the full sublevel homotopy type. In dimension two, the Chen--Lin method already applies.

The second stage, carried out in Section~\ref{sec:deformation}, proves the degree--Euler identity in Theorem~\ref{thm:degreeEuler}. Compactness of exact solutions does not give compactness of arbitrary unbounded Palais--Smale sequences. Following the deformation approach of Lucia \cite{Lucia05}, \cite[Proposition 1.1]{Lucia07} and Hor\'ak--Lucia \cite{HL}, Propositions~\ref{prop:strip}--\ref{prop:margin} construct fields that decrease two nearby-parameter energies simultaneously. The upper bound for $J_{\rho-\delta}$ and the lower target bound for $J_\rho$ give the orbit bound in \eqref{eq:normidentity}. Bounded approximate-critical compactness then gives finite, continuous hitting times. An additional nearby-parameter constraint makes the set between the relevant energy levels bounded. Adding the same distance term to both energies preserves the identity \eqref{eq:penaltyidentity} in Appendix~\ref{app:technical}; the proof of Lemma~\ref{lem:penalty} then compares the corresponding relative quotient spaces and retracts the modified quotient onto a finite-dimensional section. The Galerkin comparison in Lemma~\ref{lem:Galerkin} fixes the degree sign, and the subsequent Morse calculation uses an inward boundary field without assuming nondegeneracy of the original solutions.

The third stage, in Section~\ref{sec:topology}, identifies the negative sublevels. Carlotto and Malchiodi constructed test and concentration maps whose composition on weighted barycenters is homotopic to the identity \cite[Lemma 4.3]{CM}; as they explain at the start of their Section~3, a one-sided comparison does not establish homotopy equivalence. Here the test profiles prescribe the total weighted mass exactly and depend continuously on barycenters through collisions and vanishing coefficients. For the reverse comparison, spherical averaging produces constant traces on annuli. Rescaling the inner profiles and adding logarithmic functions on annuli gives the identities \eqref{eq:dilation}, while Proposition~\ref{prop:margin} of Section~\ref{sec:deformation} supplies an additional energy term to compensate for the bounded mass losses. Removing the original profiles contracts each suitably concentrated local family to a fixed profile. Nested ball carriers allow these contractions to extend a prescribed boundary map over one simplex at a time, including when ordinary centers approach points of $P$. The parameter choices precede subdivision of the domain complex, so refinements cannot consume the fixed energy and mass bounds. The relative extension proves injectivity as well as surjectivity on homotopy groups. The CW homotopy type of regular sublevels then gives the homotopy equivalence in Theorem~\ref{thm:reverse}.

Finally, configuration strata of the compact weighted barycenter space give the finite sum \eqref{eq:degree}. The admissible weighted support counts in \eqref{eq:barycenters} remain unchanged on each nonresonant parameter interval, and points with zero exponent cancel exactly. The sign restriction $\alpha_i\leq0$ enters the topology through $\beta_i\leq1$: collisions of support points or recentering an ordinary ball at a point of $P$ cannot increase its cost. Consequently the weighted barycenter space is compact and the carrier construction preserves its strict energy inequality. The results concern nonresonant parameters and the fixed range \eqref{eq:assumptions}.

\begin{remark}\label{rem:sharp}
The lower restriction $\alpha_i>-1$ is necessary to define the functional on all of $E$: otherwise $M(0)=\infty$. On $B_1(0)$, take $H(x)=|x|^{N\alpha}$ with $-1<\alpha\leq0$, set $\theta=1+\alpha>0$ and $\beta=\theta^{N-1}$, and, for $t>0$,
\[
u_t(x)=N\ln \frac{1+t}{1+t|x|^{q\theta}},\quad 
\rho_t=c_N\beta\Big(\frac t{1+t}\Big)^{N-1}.
\]
Direct radial integration gives
\[
M(u_t)=\frac{S_N(1+t)}{N\theta},\quad 
-\Delta_Nu_t=\rho_t\frac{He^{u_t}}{M(u_t)},\quad 
Q(u_t)=c_N\beta\int_0^t\frac{s^{N-1}}{(1+s)^N}\dd s.
\]
Since $s^{N-1}|u_t'(s)|^{N-2}u_t'(s)\to0$ as $s\downarrow0$, these identities hold weakly at the origin and $u_t\in E$. Thus
\[
\max_{x\in\overline B_1(0)}u_t(x)=N\ln (1+t)\to \infty,\quad
\rho_t\to  c_N\beta,\quad
J_{\rho,H}(u_t)=(c_N\beta-\rho)\ln  t+O(1).
\]
For fixed $t>0$, $|u_t'(s)|=Ntq\theta s^{q\theta-1}/(1+ts^{q\theta})$. It is unbounded at zero when $-1<\alpha<-1/N$; at $\alpha=-1/N$ it has a nonzero limit, so the radial function is not differentiable there. Hence global $C^1$ regularity cannot hold throughout \eqref{eq:assumptions}.

The same obstruction holds without radial symmetry. Suppose $\rho>0$ and a
solution is $C^1$ near $p_i$, and write
$H(x)=h_i(x)|x-p_i|^{N\alpha_i}$ with $h_i(p_i)>0$.
Set $\mathcal V_u=|\nabla u|^{N-2}\nabla u$, choose a nonzero
$\psi\in C^\infty_c(B_1(0))$ with $\psi\geq0$, and let
$\psi_R(x)=\psi((x-p_i)/R)$. Testing \eqref{maineq} gives, as $R\downarrow0$,
\begin{align}
    C_iR^{N(1+\alpha_i)}(1+o(1))
&=\frac{\rho}{M(u)}\int_\Om He^u\psi_R\dd x\nonumber\\
&=\int_{B_R(p_i)}[\mathcal V_u(x)-\mathcal V_u(p_i)]
       \cdot\nabla\psi_R(x)\dd x
=o(R^{N-1}),\label{eq:C1obstruction}
\end{align}
where
$C_i=\rho h_i(p_i)e^{u(p_i)}M(u)^{-1}
\int_{B_1(0)}|y|^{N\alpha_i}\psi(y)\dd y>0$.
The last equality follows from continuity of $\mathcal V_u$ and
$\int_{B_R(p_i)}\nabla\psi_R\dd x=0$.
Equation \eqref{eq:C1obstruction} is impossible when $\alpha_i\leq-1/N$.
For $\alpha>-1/N$, the displayed radial solutions are $C^1$ at zero,
so this obstruction cannot be extended to a larger exponent range.

The first threshold and logarithmic constant are sharp in this class; the functional is not coercive at $c_N\beta$ and is unbounded below above it. This does not realize every value of $\Sigma$.

The coercivity range in Proposition~\ref{prop:functional} is optimal on every domain considered here. Set $\beta_*:=\min(\{1\}\cup\{\beta_i:1\leq i\leq m\})$ and choose an interior point $y$ of weight $\beta_*$. Write $\theta=\beta_*^{1/(N-1)}$ and choose $B_R(y)\Subset\Om$ with $B_R(y)\cap P\subset\{y\}$, so that $H(x)=h_y(x)|x-y|^{N(\theta-1)}$ there, with $h_y$ positive and continuous. For $T>0$, define
\[
w_T(x)=\gamma_N\theta\min\{T,\ln(R/|x-y|)\}\quad\text{in }B_R(y),
\]
set $w_T(y)=\gamma_N\theta T$, and extend $w_T$ by zero. Direct radial integration, using the positive upper and lower bounds for $h_y$, gives
\[
Q(w_T)=q\theta c_N\beta_*T,\quad B(w_T)=q\theta T+O(1),\quad
J_{\rho,H}(w_T)=q\theta(c_N\beta_*-\rho)T+O(1).
\]
Since $\norm{w_T}_E\to\infty$, $J_{\rho,H}$ is not coercive at $\rho=c_N\beta_*$ and is unbounded below for every $\rho>c_N\beta_*$. Moreover $B(w_T)/Q(w_T)\to1/(c_N\beta_*)$, so no estimate $B(u)\leq cQ(u)+C$ can hold on $E$ with $c<1/(c_N\beta_*)$. These conclusions include resonant parameters and do not require the test-map construction. They do not assert the existence of blowing-up solutions at every resonance.

No optimality is claimed for the $C^1$ regularity of $h$ or for the smoothness of $\partial\Om$.
\end{remark}

The role of the upper restriction on the exponents is different from the
integrability restriction discussed in Remark~\ref{rem:sharp}.
\begin{remark}\label{rem:positive}
For arbitrary fixed exponents $\alpha_i>-1$, $H\in L^s(\Om)$ for some
$s>1$ and $H>0$ almost everywhere. The variational argument in
Proposition~\ref{prop:functional} still gives a $C^1$ functional, a compact
solution map, and a minimizer for $0\leq\rho<c_N\beta_*$, where
$\beta_*:=\min(\{1\}\cup\{(1+\alpha_i)^{N-1}:1\leq i\leq m\})$.
The difficulty with positive exponents is therefore not the definition of
the functional.

If $\alpha_i>0$, then $\beta_i>1$, but the local estimate
\eqref{eq:localMT} cannot retain the weight $\omega=\beta_i$ on an
arbitrary neighborhood of $p_i$. Choose an ordinary point $y\ne p_i$
in that neighborhood and $B_R(y)$ avoiding $P$. For $T>0$, set
$v_T(x)=\gamma_N\min\{T,\ln(R/|x-y|)\}$ in $B_R(y)$, with
$v_T(y)=\gamma_NT$, and extend it by zero. Since $H$ is positive
and continuous on this ball, radial integration gives
\begin{equation}\label{eq:positiveprofiles}
Q(v_T)=qc_NT,\quad  \ln M(v_T)=qT+O(1),\quad 
\mu_{v_T}\weak\delta_y\quad(T\to\infty).
\end{equation}
Equation \eqref{eq:positiveprofiles} contradicts \eqref{eq:localMT}
with $\omega=\beta_i$ whenever
$0<\varepsilon<(1-\beta_i^{-1})/c_N$; the mass asymptotic is unchanged
when the integral is restricted to the chosen neighborhood. Thus concentration near a
positive zero must be distinguished by its location and scale; the
larger weight cannot be assigned to the whole neighborhood. Related
improvements with additional concentration information in dimension two
were obtained by Malchiodi--Ruiz \cite[Propositions 3.1 and 4.1]{MR11}.

There is also a topological obstruction. With one point $p_1=0$ of
weight $\beta>1$ and $1<r<\min\{2,\beta\}$, ordinary points
$x_j\to0$ give $\delta_{x_j}\in\Br(K)$ but $\delta_0\notin\Br(K)$.
For $N=3$, take $\Om=B_2(0)$ and the polyhedral core $K=[-1,1]^3$. The space is exactly
$\{\delta_x:x\in K\setminus\{0\}\}\simeq\mathbb S^2$.
Its ordinary Euler characteristic is $2$, whereas its compactly
supported Euler characteristic is $0$. Formula \eqref{eq:degree}
would give $1$ in this example, while \eqref{eq:mainhomotopy} with
this space would give $-1$. Hence the degree formula and the sublevel
identification cannot both extend unchanged. The distinction between
$\chi$ and $\chi_c$ is essential here; see
\cite[Theorem 1.1 and Corollary 1.2]{Kallel}. Moreover, recentering an
ordinary ball at $p_i$ in Lemma~\ref{lem:carriers} changes its weight
from $1$ to $\beta_i>1$, and may violate the energy inequality in
Lemma~\ref{lem:contract}.

The compactness proof also needs a new argument at a positive zero.
The exponential estimate \eqref{eq:BM} directly gives the ordinary
small-mass threshold there, rather than the threshold with
$\theta=1+\alpha_i>1$ in \eqref{eq:smallmass}. The resulting lower
bound for an atom no longer supplies the logarithmic coefficient
$a\geq N(1+\alpha_i)$ used to prove $b=0$ in
Theorem~\ref{thm:compact}. The absence of diffuse mass must therefore
be established before repeating that proof. This should be distinguished
from the total-mass formula for entire solutions, which already holds
for all $\alpha>-1$ by \cite[Theorem 1.4]{Esposito}.
The local Pohozaev and dilation identities
\eqref{eq:interiorpoho} and \eqref{eq:dilation} do not themselves
require nonpositive exponents. In dimension two the positive-exponent
degree formula is already covered by \cite[Theorem 1.4]{CL15}; for
$N\geq3$, extending the present argument requires new concentration
and topological comparisons, rather than a change of the range in
\eqref{eq:assumptions}.
\end{remark}

For resonant parameters $\rho\in\Sigma$, we plan to study the degree by adapting Bahri's theory of critical points at infinity \cite{Bahri}. In the degenerate quasilinear setting, an appropriate nonsmooth critical point theory may be needed, as the standard smooth arguments may not apply directly.

The present paper is organized as follows. Section~\ref{sec:compact} establishes the analytic framework and concentration. Section~\ref{sec:deformation} proves the deformation and degree--Euler statements. Section~\ref{sec:topology} identifies negative sublevels, and Section~\ref{sec:formula} completes the proofs of the main results. Appendix~\ref{app:technical} contains the proofs of several technical lemmas.

\medskip
\noindent\textbf{Notation.}
Throughout the paper, we use the following notation and conventions.
We use $s'=s/(s-1)$ for $s>1$, $t^+=\max\{t,0\}$, and $\ln^+t=\max\{\ln t,0\}$ for $t>0$.
When $H$ is fixed, we write $J_\rho=J_{\rho,H}$ and
$d_\rho=d_\rho(\Om,H)$.
\begin{itemize}[leftmargin=1.5em]
\setlength{\itemsep}{4pt}

\item
$E=W^{1,N}_0(\Om)$ is equipped with
$\norm{u}_E=\norm{\nabla u}_{L^N(\Om)}$.
Its real continuous dual is $E^*$, with pairing
$\pair{\cdot}{\cdot}$, and $B_R^E$ denotes the open ball of radius
$R$ centered at zero. Balls $B_R(x)$ without the superscript $E$ are Euclidean. We denote the outward unit normal by $\nu$ and surface measure by $\dd S_x$.

\item
$Q(u)$, $M(u)$, $B(u)=\ln  M(u)$, and $J_\rho(u)$ are defined
in \eqref{eq:functional}. For $a\in\R$, we set
$J_\rho^a=\{u\in E:J_\rho(u)\leq a\}$.

\item
$q$, $\gamma_N$, $S_N$, and $c_N$ are the constants in
\eqref{eq:constants}; in particular, $c_2=8\pi$.
We use $\beta_i=(1+\alpha_i)^{N-1}$,
$\beta_I=\sum_{i\in I}\beta_i$, and $r=\rho/c_N$,
with $\beta_\varnothing=0$. The resonance set $\Sigma$ is defined
in \eqref{eq:resonance}; $\rho\geq0$ is nonresonant when
$\rho\notin\Sigma$.

\item
$A:E\to E^*$ denotes the Dirichlet $N$-Laplacian.
The solution map $K_{\rho,H}$ and the Leray--Schauder degree
$d_\rho$ are defined in \eqref{eq:solutionmap} and
\eqref{eq:degreedef}, respectively.

\item
$\Prob(Y)$ denotes the probability measures on a compact space
$Y$, with the weak topology. For $u\in E$, the normalized measure
on $\bOm$ is $\mu_u=H(x)e^{u(x)}\dd x/M(u)$.

\item
The weighted support cardinality $\card_\beta$ and the space $\Br(Y)$ are defined in Section~\ref{subsec:main}. For compact $Y\subset\bOm$, the bounded--Lipschitz distance on $\Prob(Y)$ is
\begin{equation}\label{eq:BLdistance}
d_{\mathrm{BL}}(\mu,\nu)=\sup_{\substack{f:Y\to\R\\\max_{x\in Y}|f(x)|+\Lip(f)\leq1}}\Big|\int_Y f(x)\dd(\mu-\nu)(x)\Big|.
\end{equation}
Here $\Lip(f)$ is computed using Euclidean distance. We identify measures on $Y$ with their extensions by zero to $\bOm$.

\item
In Section~\ref{sec:topology} we use $\Fcal_a=(1+a)Q-\rho B$, $a\geq0$, as in \eqref{eq:weightedenergy}; thus $\Fcal_0=J_\rho$.
The constraint set $\Mcap=\{J_{\rho-\delta}\leq C\}$, introduced in Section~\ref{sec:deformation} after $\delta,C$ have been fixed, is distinct from the mass integral $M(u)$.

\item
Homology is singular homology. When its total rank over $\Q$ is finite, $\chi(X)=\sum_{j=0}^\infty(-1)^j\dim_\Q H_j(X;\Q)$ and $\chi(X,A)=\sum_{j=0}^\infty(-1)^j\dim_\Q H_j(X,A;\Q)$. We set $\chi(\varnothing)=0$ and distinguish $\chi$ from the compactly supported Euler characteristic $\chi_c$ on noncompact configuration strata. We write $X\simeq Y$ for homotopy equivalent spaces and $f\simeq g$ for homotopic maps.

\item
The constant $C>0$ may vary from line to line. Its dependencies
are specified in the statements; dependence on the fixed domain,
weight, chosen balls, and cutoff functions may be suppressed in proofs.
Unless otherwise stated, constants in estimates for sequences are
independent of the sequence index, and $O(\cdot)$ and $o(\cdot)$
refer to the limit specified in the statement or proof.
\end{itemize}

\section{Compactness and concentration}\label{sec:compact}
The points and exponents are fixed in this section. We write $H_h(x)=h(x)\prod_{i=1}^m|x-p_i|^{N\alpha_i}$, $M_h(u)=\int_\Om H_he^u\dd x$, and $\theta_i=1+\alpha_i$, so $\beta_i=\theta_i^{N-1}$. Constants may depend on $\Om$, $P$, and the exponents. Uniformity in $h$ always means a common positive lower bound and a common $C^1(\bOm)$ bound. We use the constants in \eqref{eq:constants} and the resonance set \eqref{eq:resonance} throughout.

\subsection{Variational framework}
\begin{proposition}\label{prop:functional}
For $\rho\geq0$, $J_{\rho,H_h}$ is well defined and $C^1$ on $E$. Define
\[
\pair{A(u)}{v}=\int_\Om|\nabla u|^{N-2}\nabla u\cdot\nabla v\dd x,\quad 
\pair{\Gcal_h(u)}{v}=\frac{\int_\Om H_he^uv\dd x}{M_h(u)}.
\]
Then $J'_{\rho,H_h}=A-\rho\Gcal_h$, $A$ is a homeomorphism, and $K_{\rho,H_h}=A^{-1}(\rho\Gcal_h)$ is completely continuous. This compactness is collective for bounded $\rho$ and the preceding bounded positive families of $h$.

Every critical point is a bounded continuous weak solution of \eqref{maineq} and belongs to $C^{1,\eta}$ on compact subsets of $\bOm\setminus P$, for some $\eta>0$ on each such set. The bounds are uniform for critical points in bounded subsets of $E$. If $\beta_*:=\min\{1,\beta_1,\ldots,\beta_m\}$, then $J_{\rho,H_h}$ is coercive and attains its minimum for $0\leq\rho<c_N\beta_*$.
\end{proposition}
\begin{proof}
The Moser--Trudinger inequality and Young's inequality with conjugate exponents $q$ and $N$ imply
\begin{equation}\label{eq:expbounded}
\sup_{\norm{u}_E\leq R}\int_\Om e^{a|u|}\dd x<\infty\quad (R,a<\infty).
\end{equation}
The normalization of the sharp logarithmic form is
\begin{equation}\label{eq:MT}
\ln \int_\Om e^v\dd x\leq C_\Om+\frac{Q(v)}{c_N};
\end{equation}
Indeed, for $R>0$ and $t\geq0$, optimization of the right-hand side gives
\[
t\leq N S_N^{1/(N-1)}(t/R)^q+\frac{R^N}{Nc_N}.
\]
Apply this with $t=|v(x)|$ and $R=\norm v_E$, and then use the sharp Moser inequality \cite[Theorem 1]{Moser}; the case $R=0$ is immediate. The resulting inequality is \eqref{eq:MT} with the normalization of $Q$ in \eqref{eq:functional}. Choose $s>1$ with $H_h\in L^s$. H\"older's inequality then gives $H_he^u\in L^t$ for some $t>1$, uniformly on bounded subsets of $E$. Jensen's inequality for the probability measure $\dd\nu_h(x)=H_h(x)\dd x/\int_\Om H_h(y)\dd y$ gives
\[
M_h(u)\geq\Big(\int_\Om H_h\dd x\Big)\exp\Big(\int_\Om u(x)\dd\nu_h(x)\Big).
\]
The embedding $E\hookrightarrow L^{s/(s-1)}$ gives a positive lower bound on each bounded subset.

If $u_j\weak u$ in $E$, compact Sobolev embeddings and \eqref{eq:expbounded} give $e^{u_j}\to e^u$ in every finite $L^a$, by uniform integrability, and hence $H_he^{u_j}\to H_he^u$ in some $L^t$, $t>1$. Thus $\Gcal_h(u_j)\to\Gcal_h(u)$ in $E^*$. The exponential remainder formula, \eqref{eq:expbounded}, and H\"older's inequality prove continuous Fr\'echet differentiability. For completeness, $A$ is onto because $Q(u)-\pair{f}{u}$ is coercive, strictly convex, and weakly lower semicontinuous for $f\in E^*$. Its unique minimizer solves $A(u)=f$. The inequality
\begin{equation}\label{eq:monotone}
\pair{A(v)-A(w)}{v-w}\geq d_N\norm{v-w}_E^N\quad (d_N>0)
\end{equation}
proves injectivity and continuity of $A^{-1}$; continuity of $A$ follows from H\"older's inequality. Continuity of $A^{-1}$ and compactness of $\Gcal_h$ prove compactness of $K_{\rho,H_h}$ in \eqref{eq:solutionmap}. For varying $h_j$, take a uniformly convergent subsequence; then $H_{h_j}\to H_h$ in a fixed $L^s$ and the same argument applies.

To prove continuity without assuming regularity at $P$, we first establish the stability estimate \eqref{eq:supstability}. On a smooth bounded domain, if $A(u)=f$, $A(v)=g$, $u,v\in E$, and $f-g\in L^s$, $s>1$, then
\begin{equation}\label{eq:supstability}
\norm{u-v}_{L^\infty}\leq C_s\norm{f-g}_{L^s}^{1/(N-1)}.
\end{equation}
Indeed, for $k\geq0$, test the difference with $w=(u-v-k)^+$, using bounded truncations first. For any fixed $a>Ns'$, where $s'=s/(s-1)$, monotonicity and the Sobolev inequality yield, with $D_k=\{u-v>k\}$,
\[
\norm{w}_{L^a}^{N-1}\leq C\norm{f-g}_{L^s}|D_k|^{1/s'-1/a}.
\]
Consequently, for $h>k$,
\[
|D_h|\leq\Big(\frac{C\norm{f-g}_{L^s}^{1/(N-1)}}{h-k}\Big)^a
|D_k|^{(a/s'-1)/(N-1)}.
\]
The last exponent exceeds one. Substitution of $k_j=K(1-2^{-j})$ and induction, with $K$ a sufficiently large multiple of $\norm{f-g}_{L^s}^{1/(N-1)}$, gives $|D_{k_j}|\to0$. Interchanging $u,v$ proves \eqref{eq:supstability}.

For bounded right-hand sides, the interior and zero-Dirichlet boundary $C^{1,\eta}$ estimates used here are exactly \cite[Theorem 18]{EM}; the required local boundedness follows from \cite[Theorem 15 and Corollary 17]{EM}. Truncate an $L^t$ right-hand side and use \eqref{eq:supstability} with $1<s<t$. The corresponding solutions with bounded right-hand sides converge uniformly, so their limit is continuous on $\bOm$. For a bounded family of $L^t$ right-hand sides, the truncation tails tend to zero uniformly in $L^s$; the solutions with bounded right-hand sides have uniform $C^{1,\eta}$ bounds. The same argument proves relative compactness of the full family in $C^0(\bOm)$. These conclusions require no gradient regularity at $P$. For a critical point, $H_he^u/M_h(u)\in L^t$ as above; boundedness and continuity follow. Away from $P$ the right-hand side is bounded, and the cited $C^{1,\eta}$ estimates apply, including on the boundary, where the trace is zero.

Finally, H\"older's inequality and \eqref{eq:MT}, with $s'=s/(s-1)$, give
\[
\ln  M_h(u)\leq C_s+\frac{(s')^{N-1}}{c_N}Q(u).
\]
When some $\alpha_i<0$, admissible $s$ approach $\min_{\substack{1\leq i\leq m\\\alpha_i<0}}(-1/\alpha_i)$ from below, so $(s')^{N-1}\to1/\beta_*$. If all exponents are nonnegative, $H_h$ is bounded and the coefficient is directly $1/c_N$. For $\rho<c_N\beta_*$ an admissible $s$ makes $J_{\rho,H_h}\geq cQ-C$, $c>0$. Weak continuity of the exponential term on bounded sets and weak lower semicontinuity of $Q$ give a minimizer.
\end{proof}

\subsection{Blow-up analysis}
Set $a_N=N S_N^{1/(N-1)}$. If $-\Delta_Nw=f\geq0$ on a smooth bounded domain $D$ and $\varphi$ is the $N$-harmonic function with the same boundary values, then \cite[Lemma 4]{EM} gives
\begin{equation}\label{eq:BM}
\int_D\exp\Big(\frac{a|w-\varphi|}{\norm f_{L^1(D)}^{1/(N-1)}}\Big)\dd x\leq C(a,D),
\quad  0<a<d_N^{1/(N-1)}a_N.
\end{equation}
For zero boundary values one may replace $d_N^{1/(N-1)}a_N$ by $a_N$. The constant $d_N$ is as in \eqref{eq:monotone}; the convention $S_N=N|B_1|$ gives the same constants as in \cite{EM}. We use only strictly subcritical coefficients. The estimate also applies to energy solutions with right-hand side in $L^t$, $t>1$, and a bounded $N$-harmonic function in $W^{1,N}$ with the same boundary values. To see this, approximate the right-hand side in $L^t$ and its boundary extension in $W^{1,N}$ by smooth data. For zero boundary values use zero extensions throughout. For nonconstant data, approximate a bounded extension by bounded smooth functions and solve the corresponding $N$-harmonic Dirichlet problems. Subtracting the weak equations, using the difference of the boundary extensions as a lifting, and applying \eqref{eq:monotone} shows strong $W^{1,N}$ convergence of both the solutions and the corresponding $N$-harmonic functions: the lifting terms tend to zero by H\"older's inequality, while the error in the right-hand side tends to zero in $(W^{1,N}_0)^*$. In this argument the energies are bounded by the minimization problem with the same lifting. The constant in \eqref{eq:BM} is independent of the boundary data. Fatou's theorem therefore gives the same estimate in the limit for every coefficient strictly below its threshold. If $f=0$, the solution equals $\varphi$.

\begin{lemma}\label{lem:measure}
Let $D$ be smooth and bounded and let $\varphi\in W^{1,N}(D)\cap L^\infty(D)$ be $N$-harmonic. Suppose $z_j-\varphi\in W^{1,N}_0(D)$ and
\[
-\Delta_Nz_j=f_j\geq0,\quad  \sup_{j\in\N}\norm{f_j}_{L^1(D)}\leq M<\infty,
\]
where each $f_j\in L^{t_j}(D)$ for some $t_j>1$. A subsequence converges strongly in $W^{1,a}(D)$ for every $1\leq a<N$, and $|\nabla z_j|^{N-2}\nabla z_j$ converges strongly in $L^1(D;\R^N)$. If $f_j\dd x\weak\mu$ as measures, the limit $Z$ satisfies $-\Delta_NZ=\mu$ in distributions and has the prescribed trace.
\end{lemma}
\begin{proof}
    The proof is in Appendix~\ref{app:technical}.
\end{proof}

\begin{lemma}\label{lem:boundaryremoval}
Let $D$ be smooth and bounded and $p\in\partial D$. Suppose
\[
U\geq0,\quad U\in\bigcap_{1<a<N}W^{1,a}(D),\quad U\in C^{1,\eta_0}_{loc}(\overline D\setminus\{p\}),
\]
for some $\eta_0>0$, with zero trace on a boundary neighborhood of $p$, and
\[
-\Delta_NU=F=bWe^U\quad\text{ in }\, \mathcal D'(D),\quad  F\in L^1(D),\quad b\geq0,
\]
where $W\geq0$ is bounded near $p$. Then $U$ extends through $p$ as a $C^{1,\eta}$ function on a smaller closed neighborhood, for some $\eta>0$.
\end{lemma}
\begin{proof}
Choose a smooth $D_0\subset D$ agreeing with $D$ near $p$ and with $\partial D_0\cap D$ away from $p$. Its boundary data $g$ are $C^{1,\eta_0}$: they vanish on the boundary portion near $p$ and otherwise come from $U$ away from $p$. They admit a compactly supported extension $\widetilde g\in W^{1,N+\varepsilon}(\R^N)$ for some $\varepsilon>0$.

Fix an exponent $a>N$ and shrink $D_0$ so that
\[
a\norm F_{L^1(D_0)}^{1/(N-1)}<d_N^{1/(N-1)}a_N.
\]
Let $z_j$ be the energy solution with boundary value $g$ and right-hand side $F_j=\min(F,j)$, and let $\varphi$ be the $N$-harmonic energy solution with the same boundary value. Minimization gives $\varphi\in W^{1,N}(D_0)$, and comparison with the extrema of $g$ bounds it. Thus Lemma~\ref{lem:measure} applies with this choice of $\varphi$; no gradient estimate for nonconstant boundary data is needed. Estimate \eqref{eq:BM} gives
\begin{equation}\label{eq:boundaryexp}
\sup_{j\in\N}\int_{D_0}e^{a|z_j-\varphi|}\dd x<\infty.
\end{equation}
Lemma~\ref{lem:measure} gives a limit $Z$, with $z_j\to Z$ almost everywhere and in every $W^{1,t}$, $t<N$, satisfying $-\Delta_NZ=F$ and having trace $g$.

Here the identification $Z=U$ requires a uniqueness result for solutions below the energy class. We apply \cite[Proposition 6.2, p.~32]{DHM}. A smooth bounded domain has the complement-density property~A used there: exterior boundary cones give a uniform positive volume fraction at small radii, compactness handles the remaining bounded radii, and the boundedness of the domain handles large radii. For the scalar field $\sigma(x,\xi)=|\xi|^{N-2}\xi$, hypotheses (H0), (H2), and (H3) of that paper hold with $g_1=g_2=1$ and $g_3=g_4=g_5=g_6=0$. In particular its structural test is $(1-a^2)|\xi|^N\geq0$ for $|a|\leq1$. The field is independent of the unknown and satisfies the required uniform $N$-monotonicity. If $p_0\in(N-1,N)$ is the threshold in that proposition, both $U,Z$ belong to $W^{1,t}(D_0)$ for a fixed $p_0<t<N$, have a finite-measure distributional $N$-Laplacian, and have the extension $\widetilde g$ just constructed. The proposition gives $\nabla U,\nabla Z\in L^{N,\infty}(D_0)$, where $f\in L^{N,\infty}$ means $\sup_{s>0}s\,|\{|f|>s\}|^{1/N}<\infty$. Now \cite[Theorem 1.2, p.~3]{DHM} applies: their distributional $N$-Laplacians agree, $U-Z\in W^{1,1}_0(D_0)$, and both gradients are weak-$L^N$. Thus $U=Z$.

The identification uses uniqueness for the fixed right-hand side $F$, not uniqueness for the nonlinear equation $-\Delta_NU=bWe^U$. The equality $Z=U$ is therefore established before improving the regularity of $U$. Fatou's theorem and \eqref{eq:boundaryexp}, together with the boundedness of $\varphi$, now give $e^U\in L^a(D_0)$ and hence $F\in L^a(D_0)$. Choose the domain inside the neighborhood where $W$ is bounded. The right-hand sides $F_j$ are uniformly bounded in $L^a$, and \eqref{eq:boundaryexp} bounds $z_j$ in $L^N$.

Apply \cite[Theorem 15]{EM} on a smaller neighborhood meeting only the boundary portion where $g=0$. Here $z_j\geq0$: the data $g$ and $F_j$ are nonnegative, so comparison applies. Thus the constant-trace and sign hypotheses of that boundary estimate hold. It bounds $z_j$, and consequently $U$, there. Hence $F=bWe^U$ is bounded on this smaller neighborhood. Since $0\leq F_j\leq F$, \cite[Theorem 18]{EM}, again with zero trace on the actual boundary portion, gives uniform local $C^{1,\eta}$ bounds for $z_j$. A subsequence converges in $C^1$ on a still smaller closed neighborhood to the already identified limit $U$. This proves the extension through $p$. If $b=0$, the same argument applies with every $F_j=0$.
\end{proof}

\begin{theorem}\label{thm:compact}
Let $\Lambda\subset[0,\infty)\setminus\Sigma$ be compact and $C_0>1$. There exists a constant $C<\infty$ such that every weak solution with
\[
\rho\in\Lambda,\quad  C_0^{-1}\leq h\leq C_0,\quad  \norm{\nabla h}_{L^\infty}\leq C_0
\]
satisfies $\norm u_{L^\infty(\Om)}+\norm u_E\leq C$. These solutions are relatively compact in $E\cap C^0(\bOm)$. For fixed $h$, the pairs $(\rho,u)$ form a compact subset of $\Lambda\times E$ and of $\Lambda\times C^0(\bOm)$.
\end{theorem}
\begin{proof}
Suppose $\norm{u_j}_\infty\to\infty$. Extract $\rho_j\to\rho_\infty\in\Lambda$, $h_j\to h_\infty$ uniformly, and omit the terms $\rho_j=0$, for which $u_j=0$. Comparison gives $u_j\geq0$. Set
\[
b_j=\frac{\rho_j}{\int_\Om H_je^{u_j}\dd x},\quad v_j=u_j+\ln  b_j,\quad f_j=H_je^{v_j},\quad H_j=H_{h_j}.
\]
Then
\begin{equation}\label{eq:normalized}
-\Delta_Nu_j=-\Delta_Nv_j=f_j,\quad \int_\Om f_j\dd x=\rho_j,\quad
u_j|_{\partial\Om}=0,\quad v_j|_{\partial\Om}=\ln  b_j.
\end{equation}
Since $u_j\geq0$ and $\int_\Om H_j\dd x\geq c>0$, we have $0<b_j\leq\rho_j/\int_\Om H_j\dd x\leq C$. Extract $b_j\to b\geq0$ and $f_j\dd x\weak\mu$ on $\bOm$. The zero-boundary version of \eqref{eq:BM} gives, for some $\delta>0$,
\begin{equation}\label{eq:globalexp}
\sup_{j\in\N}\int_\Om e^{\delta u_j}\dd x<\infty,\quad 
\sup_{j\in\N}\norm{u_j}_{L^a}<\infty\quad(1\leq a<\infty).
\end{equation}
In particular $v_j^+$ is bounded in $L^N$.

Near $p_i$ write $H_j=|x-p_i|^{N\alpha_i}h_{i,j}$ with uniformly positive $C^1$ factors. Set $\theta=1$ at an ordinary point and $\theta=\theta_i$ at $p_i$. The small-mass condition
\begin{equation}\label{eq:smallmass}
\limsup_{j\to\infty}\int_D f_j\dd x<d_NS_N(N\theta)^{N-1}
\end{equation}
implies uniform upper bounds for $v_j$ and $u_j$ on a smaller set.
Here $D$ may be a small ball, or a smooth local domain agreeing with $\Om$ near a boundary point. To prove it, choose $a>1/\theta$ with
\[
a\Big(\limsup_{j\to\infty}\int_D f_j\dd x\Big)^{1/(N-1)}<d_N^{1/(N-1)}a_N.
\]
There exists $s>1$ for which $H_j$ is bounded in $L^s(D)$ and $1/s+1/a<1$. Let $\varphi_j$ be the $N$-harmonic function with the same boundary values as $v_j$. It lies below $v_j$ by comparison, so its positive part has bounded $L^N$ norm. The local upper bound \cite[Theorem 15]{EM} bounds $\varphi_j$ on a smaller ball. At a boundary point, $\varphi_j$ has constant trace $\ln  b_j$ on the actual boundary and is at least that constant by comparison. These are the local boundary hypotheses of \cite[Theorem 15]{EM}; its proof flattens and reflects only this boundary portion, so it imposes no condition on the remaining boundary outside the smaller set. The fixed upper bound on $\ln  b_j$ and the $L^N$ bound on the positive part give a uniform local upper estimate. Estimate \eqref{eq:BM} bounds $e^{v_j}$ in $L^a$ locally. H\"older's inequality bounds $f_j$ in some $L^t$, $t>1$. The local estimate cited above bounds $v_j^+$, and \eqref{eq:globalexp} bounds $u_j$ on a further smaller set.

Choose $\eta_0>0$ smaller than all the right-hand thresholds in \eqref{eq:smallmass}, and set
\[
S=\{x\in\bOm:\mu(\{x\})\geq\eta_0\}.
\]
The set $S$ is finite because $\mu$ has finite mass. At any other point, continuity from above gives a neighborhood with mass strictly below $\eta_0$; choose its spherical boundary to have zero $\mu$-mass so that \eqref{eq:smallmass} applies to the approximating measures. Conversely, a positive point mass forces $v_j$ to be unbounded above in every neighborhood, since a local upper bound would imply $f_j\leq CH_j$ and uniform integrability. Thus $S$ is exactly the concentration set. On compact subsets of $\bOm\setminus(S\cup P)$ the right-hand sides are bounded, and
\begin{equation}\label{eq:C1away}
u_j\to  U\quad\text{ in }\, C^1_{loc}(\bOm\setminus(S\cup P)),
\end{equation}
with uniform local $C^{1,\eta}$ bounds. Lemma~\ref{lem:measure}, with $\varphi=0$, also gives strong $W^{1,a}$ convergence for every $a<N$ and strong $L^1$ convergence of $|\nabla u_j|^{N-2}\nabla u_j$.

For an interior $p\in S$ the concentration-mass bound improves to
\begin{equation}\label{eq:atomlower}
\mu(\{p\})\geq S_N(N\theta)^{N-1},\quad 
\theta=\left\{\begin{aligned}
    &1+\alpha_i&&\text{ if }\, p=p_i,\\&1&& \text{ if }\, p\notin P.
\end{aligned}\right.
\end{equation}
Choose a ball about $p$ containing no other point of $S\cup P$. Its boundary has a uniform upper bound for $v_j$. The zero-boundary solution $z_j$ with right-hand side $f_j$ there satisfies $v_j\leq z_j+C$. If \eqref{eq:atomlower} failed, the sharper coefficient $a_N$ in \eqref{eq:BM} would allow $a>1/\theta$ on a small ball. The preceding H\"older's inequality and local boundedness argument would bound $v_j$ at $p$, a contradiction.

We next exclude $p\in S\cap\partial\Om$. Choose a smooth local domain $D$ agreeing with $\Om$ near $p$, with $\overline{\partial D\cap\Om}$ disjoint from $S\cup P$ and $\overline D$ containing no other point of $S\cup P$. The weight is uniformly positive and $C^1$ on this neighborhood. The limit $U$ has zero trace on its boundary portion in $\partial\Om$, belongs to every $W^{1,a}(D)$, $a<N$, and satisfies
\[
-\Delta_NU=bH_{h_\infty}e^U:=F\quad\text{ in }\, \mathcal D'(D).
\]
Fatou's theorem gives $F\in L^1(D)$. The point $p\in\partial D$ contributes no interior distribution. Lemma~\ref{lem:boundaryremoval} gives $C^{1,\eta}$ regularity through $p$.

Choose $R_0>0$ so that $\overline\Om\cap\overline B_{2R_0}(p)$ contains no point of $(S\setminus\{p\})\cup P$ and $U$ is $C^{1,\eta}$ on this closed neighborhood. For a transverse radius $0<R<R_0$, write
\[
D_R=\Om\cap B_R(p),\quad 
\Gamma_R=\partial\Om\cap B_R(p),\quad 
S_R=\Om\cap\partial B_R(p).
\]
All constants below are independent of $j$ and small $R$, except the bounds used to smooth the edge, which are used only at fixed $j,R$. Set
\[
t_{j,R}=\frac{\int_{\Gamma_R}(x-p)\cdot\nu(x)|\nabla u_j|^N\dd S_x}
{\int_{\Gamma_R}\nu(p)\cdot\nu(x)|\nabla u_j|^N\dd S_x},
\quad  y_{j,R}=p+t_{j,R}\nu(p).
\]
If $\int_{\Gamma_R}|\nabla u_j|^N\dd S_x=0$, set $t_{j,R}=0$. For small $R$, $\nu(p)\cdot\nu(x)>1/2$; thus in that exceptional case $|\nabla u_j|=0$ almost everywhere on $\Gamma_R$. For $x\in\Gamma_R$, the boundary graph expansion gives $|(x-p)\cdot\nu(x)|\leq CR^2$. Consequently $|t_{j,R}|\leq CR^2$ in both cases. In either case,
\begin{equation}\label{eq:capcancel}
\int_{\Gamma_R}(x-y_{j,R})\cdot\nu(x)|\nabla u_j|^N\dd S_x=0.
\end{equation}

To apply Pohozaev on $D_R$, fix $j$ and a transverse $R$. Smooth its codimension-two edge in a strip of width $\varepsilon$, leaving the two boundary portions unchanged elsewhere. Local product coordinates give an added boundary area $O_R(\varepsilon)$; the removed boundary strips also have area tending to zero. Since the cap avoids $P$, the fixed functions $u_j\in C^{1,\eta}$ and $H_j\in C^1$ make all boundary integrands bounded. Applying \cite[Lemma 6]{EM} with center $y_{j,R}$ and letting $\varepsilon\downarrow0$ therefore gives the identity on the original cap, with no edge term. This passage requires no bound uniform in $j$; it precedes the limits $j\to\infty$ and then $R\downarrow0$.

For $y=y_{j,R}$ define the stress contribution
\[
\mathscr S_j(x,y)=|\nabla u_j|^{N-2}
 ((x-y)\cdot\nabla u_j)\partial_\nu u_j
 -\frac1N|\nabla u_j|^N(x-y)\cdot\nu.
\]
The Pohozaev identity on $D_R$ reads
\begin{equation}\label{eq:capidentity}
\int_{D_R}\bigl[N+(x-y)\cdot\nabla\ln  H_j\bigr]f_j\dd x
=\int_{\Gamma_R\cup S_R}
 \bigl[f_j(x-y)\cdot\nu+\mathscr S_j(x,y)\bigr]\dd S_x.
\end{equation}
On $\Gamma_R$, zero trace and $C^1$ regularity give
$\nabla u_j=(\partial_\nu u_j)\nu$, and hence
\[
\mathscr S_j(x,y)=\frac{N-1}{N}|\nabla u_j|^N(x-y)\cdot\nu.
\]
The integral of $\mathscr S_j$ over $\Gamma_R$ vanishes by \eqref{eq:capcancel}; no bound on the boundary gradients uniform in $j$ is needed. On $\Gamma_R$, $f_j=b_jH_j\leq C$ and $|(x-y)\cdot\nu|\leq CR^2$, so the integral of $f_j(x-y)\cdot\nu$ is $O(R^{N+1})$. On $S_R$, $|x-y|\leq CR$ bounds the absolute value of the remaining integrands. Since $|\nabla\ln  H_j|\leq C$ in the collar, \eqref{eq:capidentity} yields
\begin{equation}\label{eq:capestimate}
(N-CR)\int_{D_R}f_j\dd x
\leq CR^{N+1}+CR\int_{S_R}(f_j+|\nabla u_j|^N)\dd S_x.
\end{equation}

We take the mass limit as measures on $\bOm$. Specifically,
\[
\int_{D_R}f_j\dd x=(f_j\dd x)(\bOm\cap B_R(p))
 \to \mu(\bOm\cap B_R(p))
\]
whenever $\mu(\bOm\cap\partial B_R(p))=0$. The equality holds because each approximating measure gives zero mass to $\partial\Om$; the limit on the right retains a possible concentration mass at $p$. The relative boundary of this continuity set in $\bOm$ lies on $\partial B_R(p)$, not on $\Gamma_R$. Excluding the at most countably many radii with positive spherical mass, and the nontransverse radii, leaves radii tending to zero.

For each such fixed $R$, \eqref{eq:C1away} gives $C^1$ convergence on the closure of $S_R$, including its edge, and $f_j\to bH_{h_\infty}e^U$ uniformly there. The regularity of $U$ through $p$ bounds this limit and $|\nabla U|$ by constants independent of small $R$. Since $|S_R|\leq CR^{N-1}$, passage to $j\to\infty$ in \eqref{eq:capestimate} gives
\[
(N-CR)\mu(\bOm\cap B_R(p))\leq CR^N.
\]
Let $R\downarrow0$ along these radii. Continuity from above of the finite measure $\mu$ yields $N\mu(\{p\})\leq0$, a contradiction. Thus $S\subset\Om$. The argument has used the order: smooth the edge at fixed $j,R$, take $j\to\infty$ at fixed $R$, and finally take $R\downarrow0$.

If $S=\varnothing$, a finite cover using \eqref{eq:smallmass} gives a global $L^t$ bound for the right-hand sides, $t>1$, and bounds $u_j$, contrary to our choice. Hence $S\ne\varnothing$. We claim $b=0$. If $b>0$, fix $p\in S$ and $\sigma=\mu(\{p\})$. A diagonal choice $r_j\downarrow0$ gives
\[
g_j:=\mathbf1_{B_{r_j}(p)}f_j,\quad g_j\dd x\weak\sigma\delta_p,\quad \int_\Om g_j\dd x\to\sigma.
\]
On a ball $B$ containing no other point of $S\cup P$, let $z_j$ have right-hand side $g_j$ and zero boundary values. Comparison gives $0\leq z_j\leq u_j$. Lemma~\ref{lem:measure} and local regularity yield a nonnegative limit $Z$ with $-\Delta_NZ=\sigma\delta_p$. The isolated-singularity result \cite[Theorem 19]{EM}, which combines the results of Serrin \cite{Serrin} and Kichenassamy--V\'eron \cite{KV}, applies: $Z$ is nonnegative and $N$-harmonic off $p$, and its distributional mass determines the coefficient. Thus
\[
Z(x)=a\ln \frac1{|x-p|}+O(1),\quad  a=(\sigma/S_N)^{1/(N-1)}\geq N\theta.
\]
But
\[
\int_B H_je^{z_j}\dd x\leq\int_\Om H_je^{u_j}\dd x=\rho_j/b_j\leq C.
\]
Fatou's theorem would make the limiting integrand integrable, whereas near $p$ it is bounded below by a positive multiple of $|x-p|^{N\alpha-a}$, nonintegrable for $a\geq N(1+\alpha)$. Here $\alpha=\alpha_i$ at a point of $P$ and $0$ otherwise. Hence $b=0$.

The local bounds now show $v_j\to-\infty$ locally uniformly on $\bOm\setminus S$. At a point of $P$ outside $S$ this follows from the local $L^t$ estimate for the right-hand side, without gradient regularity. Therefore
\begin{equation}\label{eq:atomic}
\mu=\sum_{p\in S}\sigma_p\delta_p,\quad  \sigma_p>0,
\end{equation}
with no boundary loss. Strong $L^1$ convergence of $|\nabla u_j|^{N-2}\nabla u_j$ gives $-\Delta_NU=\sum_{p\in S}\sigma_p\delta_p$ in distributions. To extend the $N$-harmonicity through a point $p_i\notin S$, use a cutoff $\zeta$ in $B_{2s}(p_i)\Subset\Om\setminus S$ and test with $\zeta^Nu_j$. The local bound for $u_j$ and Young's inequality with conjugate exponents $q$ and $N$ give
\[
\int_{B_{2s}(p_i)}\zeta^N|\nabla u_j|^N\dd x\leq C\int_{B_{2s}(p_i)}f_j\dd x+C\int_{B_{2s}(p_i)}|\nabla\zeta|^N\dd x.
\]
Thus $U\in W^{1,N}(B_s(p_i))$. Also $f_j=b_jH_je^{u_j}\to0$ in some $L^t$, $t>1$. Local monotonicity, tested against a cutoff times $u_j-U$, identifies the limit across $p_i$. Hence $U$ is $N$-harmonic on $\Om\setminus S$.

Applying \cite[Theorem 19]{EM} to the nonnegative function $U$ gives, near $p\in S$,
\begin{equation}\label{eq:asymptotic}
U(x)=a_p\ln \frac1{|x-p|}+O(1),\quad
\sigma_p=S_Na_p^{N-1},\quad
|x-p|\nabla U(x)=-a_p\frac{x-p}{|x-p|}+o(1).
\end{equation}
The asymptotics in \eqref{eq:asymptotic} concern a single nonlinear solution, not a superposition of Green functions.

Choose $B_R(p)$ avoiding the other points of $P$ and concentration points and write $H_j=|x-p|^{N\alpha}h_{p,j}$, with uniformly positive $C^1$ factors. For fixed $j$, apply \cite[Lemma 6]{EM} on $B_R(p)\setminus\overline B_\varepsilon(p)$. Since $u_j\in W^{1,N}$ and $f_j\in L^1$, there are $\varepsilon_k\downarrow0$ with
\[
\varepsilon_k\int_{\partial B_{\varepsilon_k}(p)}(|\nabla u_j|^N+f_j)\dd S_x\to 0.
\]
Otherwise radial integration would diverge logarithmically. All inner-boundary terms are bounded by this expression. The identity $(x-p)\cdot\nabla\ln  H_j=N\alpha+(x-p)\cdot\nabla\ln  h_{p,j}$ makes the volume term integrable. Passing to the puncture gives
\begin{align}
  &\,\int_{B_R(p)}[N(1+\alpha)+(x-p)\cdot\nabla\ln  h_{p,j}]f_j\dd x\nonumber\\
=&\,R\int_{\partial B_R(p)}\Big[f_j+|\nabla u_j|^{N-2}(\partial_\nu u_j)^2-\frac1N|\nabla u_j|^N\Big]\dd S_x.  \label{eq:interiorpoho}
\end{align}
In \eqref{eq:interiorpoho}, first let $j\to\infty$ at fixed $R$, and then let $R\downarrow0$. The coefficient error is bounded by $CR\int_{B_R(p)}f_j\dd x$, and $f_j\to0$ on the outer sphere. Equation \eqref{eq:asymptotic} gives
\[
N(1+\alpha)\sigma_p=\frac{N-1}{N}S_Na_p^N,\quad  \sigma_p=S_Na_p^{N-1}.
\]
Thus
\begin{equation}\label{eq:quantization}
a_p=\gamma_N(1+\alpha),\quad  \sigma_p=c_N(1+\alpha)^{N-1}.
\end{equation}
The concentration mass in \eqref{eq:atomic} was defined before rescaling, so \eqref{eq:quantization} determines its complete mass, including collapsed clusters and intermediate regions.

By \eqref{eq:normalized}, measure convergence on $\bOm$ gives $\mu(\bOm)=\rho_\infty$. Equations \eqref{eq:atomic}--\eqref{eq:quantization} imply $\rho_\infty=c_N(\ell+\beta_I)\in\Sigma$, a contradiction. Hence $\norm{u}_{L^\infty(\Om)}$ is uniformly bounded. Testing with $u$ gives $\norm u_E^N\leq\rho\norm{u}_{L^\infty(\Om)}$. Proposition~\ref{prop:functional} gives precompactness in $E$, and its truncation argument for the right-hand sides gives precompactness in $C^0(\bOm)$, since the right-hand sides are uniformly bounded in some $L^s$, $s>1$. Passing to the equation gives closedness for fixed $h$.
\end{proof}

\subsection{Low-sublevel concentration}
\begin{lemma}\label{lem:improvedMT}
Let $s\in\N$, and let $A_1,\ldots,A_s$ have pairwise disjoint relative open neighborhoods $U_j$ in $\bOm$, with smooth cutoffs equal to one on $A_j$ and supported in $U_j$. Each $\overline U_j$ either avoids $P$, with cost $\omega_j=1$, or meets $P$ at exactly one point $p_i\in U_j$, with cost $\omega_j=\beta_i$. Set $W=\sum_{j=1}^s\omega_j$. For every $\varepsilon,\eta>0$, there exists a constant $C>0$ such that
\[
\ln  M(u)\leq\Big(\frac1{c_NW}+\varepsilon\Big)Q(u)+C
\]
whenever $u\in E$ and $\int_{A_j}He^u\dd x\geq\eta M(u)$ for $1\leq j\leq s$. Boundary neighborhoods are allowed.
\end{lemma}
\begin{proof}
For a neighborhood $U$ of the stated kind and $z\in E$, H\"older's inequality and \eqref{eq:MT} give, for every $\delta>0$,
\begin{equation}\label{eq:localMT}
\ln \int_U He^z\dd x\leq\Big(\frac1{c_N\omega}+\delta\Big)Q(z)+C_\delta.
\end{equation}
For an ordinary neighborhood $H$ is bounded and $\omega=1$. At a point $p_i$ with $\alpha_i<0$, write $H=h_i|x-p_i|^{N\alpha_i}$, choose $1<s<-1/\alpha_i$, and apply \eqref{eq:MT} to $s'z$. The coefficient $(s')^{N-1}/c_N$ approaches $1/(c_N\beta_i)$. Points with zero exponent impose no restriction. Zero extension across the boundary proves the same statement there.

No Hilbert-space splitting is needed for localization. Set $R=\norm u_E$, $T=AR$, and $w=(u-T)^+$, where $A>1$. If $R=0$ the assertion is immediate. The $L^{2N}$ Sobolev embedding gives
\[
\norm w_{L^N}^N\leq\int_{\{x\in\Om:u(x)>T\}}|u|^N\dd x\leq T^{-N}\norm u_{L^{2N}}^{2N}\leq CA^{-N}R^N.
\]
For the prescribed cutoffs $\zeta_j$, $u\leq T+\zeta_jw$ on $A_j$ and
\[
\sum_{j=1}^sQ(\zeta_jw)\leq(1+\delta_1)Q(u)+C_{\delta_1}A^{-N}Q(u),
\]
by disjointness of their supports. The mass fractions and \eqref{eq:localMT} imply
\[
\omega_j\ln  M(u)\leq(1/c_N+\delta\omega_j)Q(\zeta_jw)+\omega_jAR+C.
\]
Sum, choose $\delta,\delta_1$ small and then $A$ large, and use Young's inequality with conjugate exponents $N$ and $q$ to bound $AR$ by an arbitrarily small multiple of $R^N$ plus a constant. Division by $W$ proves the estimate, with all choices independent of $u$.
\end{proof}

For a compact set $P\subset Y\subset\bOm$, recall the space $\Br(Y)$ from \eqref{eq:barycenters}.
For fixed $r$ with $c_Nr\notin\Sigma$, only finitely many admissible costs occur. The number of support points is bounded, and cost cannot increase under collision or disappearance, since $\beta_i\leq1$. Thus $\Br(Y)$ is compact. The weak topology is metrized by $d_{\mathrm{BL}}$ in \eqref{eq:BLdistance}.
\begin{lemma}\label{lem:concentration}
For nonresonant $\rho>0$, set $J=J_{\rho,H}$, $r=\rho/c_N$, and $\mu_u=He^u\dd x/M(u)$. If $\Br(\bOm)\ne\varnothing$, then
\[
\sup_{J(u)\leq-L}\dist_{\mathrm{BL}}(\mu_u,\Br(\bOm))\to 0\quad(L\to\infty).
\]
If $\Br(\bOm)=\varnothing$, $J$ is coercive and bounded below.
\end{lemma}
\begin{proof}
For fixed $\varepsilon_0>0$, the measures at distance at least $\varepsilon_0$ from $\Br(\bOm)$ form a compact set. Each has finitely many distinct support points of total cost greater than $r$: choose sufficiently many ordinary points for infinite support, and use nonmembership and nonresonance for finite support. Choose disjoint relative balls about these points, with larger neighborhoods and cutoffs as in Lemma~\ref{lem:improvedMT}. Their masses exceed $2\eta_\nu>0$. The conditions $\mu(A_j)>\eta_\nu$ define a weakly open neighborhood because $A_j$ is relatively open. A finite subcover provides finitely many mass fractions and geometries. For each choose $\varepsilon$ so that $\rho(1/(c_NW)+\varepsilon)<1$. Lemma~\ref{lem:improvedMT} then bounds $J$ below uniformly on the excluded set. If the barycenter space is empty, $r<\beta_*$ and Proposition~\ref{prop:functional} gives coercivity.
\end{proof}

\section{Deformation and degree}\label{sec:deformation}
Fix $h$ and write $J_\lambda=Q-\lambda B$, with $B=\ln  M$. All deformations and degree comparisons take place in $E$. Theorem~\ref{thm:compact} gives
\begin{equation}\label{eq:familycompact}
\{(\lambda,u)\in I\times E:J_\lambda'(u)=0\}\text{ compact in }I\times E
\end{equation}
for every compact interval $I\Subset(0,\infty)\setminus\Sigma$. Set $J_\lambda^s=\{J_\lambda\leq s\}$.

\subsection{Orbitwise deformations}
\begin{lemma}\label{lem:Splus}
The derivative $B':E\to E^*$ sends weakly convergent sequences to strongly convergent sequences, $A=Q'$ is a homeomorphism, and $K_\lambda=A^{-1}(\lambda B')$ is compact on bounded sets. Moreover $J_\lambda'$ has the $(S)_+$ property: if $u_j\weak u$ and $\limsup_{j\to\infty}\pair{J_\lambda'(u_j)}{u_j-u}\leq0$, then $u_j\to u$ in $E$. A bounded sequence with $\lambda_j$ in a compact interval of $(0,\infty)$ and $J_{\lambda_j}'(u_j)\to0$ has a strongly convergent subsequence to a critical point for a limiting parameter.
\end{lemma}
\begin{proof}
The first assertions are Proposition~\ref{prop:functional}. The compactness of $B'$ makes its term in the $(S)_+$ pairing tend to zero. Subtract $\pair{A(u)}{u_j-u}\to0$ and apply \eqref{eq:monotone}. For the last assertion extract $\lambda_j\to\lambda$ and $u_j\weak u$; then $A(u_j)=\lambda_jB'(u_j)+o(1)$ converges strongly in $E^*$ and continuity of $A^{-1}$ concludes.
\end{proof}

Excision and \eqref{eq:familycompact} make $d_\rho$ well defined, constant on nonresonant components, and invariant under bounded positive $C^1$ homotopies of $h$. No global $C^1$ regularity at the points of $P$ is required.

\begin{lemma}\label{lem:commonfield}
Let $I=[\lambda_-,\lambda_+]\Subset(0,\infty)$ and $D\subset E$ be open, with $J_\lambda'(u)\ne0$ for every $(\lambda,u)\in I\times D$. There exists a locally Lipschitz field $W:D\to E$ such that, for every $\lambda\in I$,
\begin{equation}\label{eq:commonfield}
\norm{W(u)}_E\leq1,\quad J_\lambda'(u)W(u)\leq-\tfrac14m(u),\quad
m(u)=\min_{\lambda\in I}\norm{J_\lambda'(u)}_{E^*}>0.
\end{equation}
For every set $C\subset D$ closed in $E$ and every $R<\infty$, $m$ has a positive lower bound on $C\cap\overline B_R^E$ if this set is nonempty.
\end{lemma}
\begin{proof}
The function $m$ is continuous. At fixed $u$, the derivatives form the compact convex segment $[J_{\lambda_-}'(u),J_{\lambda_+}'(u)]$, at distance $m(u)$ from zero. Hahn--Banach separation and reflexivity give $v_u\in E$, $\norm{v_u}\leq1$, with $J_\lambda'(u)v_u\leq-m(u)$ for all $\lambda\in I$. Continuity, including in the compact parameter interval, preserves $J_\lambda'(v)v_u\leq-m(v)/4$ on a neighborhood $U_u$. A locally finite locally Lipschitz partition of unity on the metric space $D$, subordinate to these neighborhoods, gives $W$ as a convex combination of these constant vectors. Such partitions are obtained from a locally finite refinement by normalizing distance-to-complement functions, so they require no smoothness of the Banach norm. Convex combination preserves both inequalities. If the last claim failed, Lemma~\ref{lem:Splus} would turn a bounded sequence with $m\to0$ into a critical point in $C\subset D$, a contradiction.
\end{proof}

\begin{proposition}\label{prop:strip}
Fix nonresonant $\rho>0$ and $a<b$. If $J_\rho$ has no critical point in $\{a\leq J_\rho\leq b\}$, then $J_\rho^a$ is a strong deformation retract of $J_\rho^b$. The deformation preserves every sublevel in its domain.
\end{proposition}
\begin{proof}
Choose a compact nonresonant parameter interval about $\rho$. Its exact critical sets have compact union. Thus some $\varepsilon>0$ excludes critical points of $J_\rho$ from $[a-2\varepsilon,b+2\varepsilon]$. For sufficiently small $\delta\in(0,\rho)$,
\[
J_\lambda'(u)\ne0\quad\text{if }\lambda\in[\rho-\delta,\rho],\quad a-\varepsilon\leq J_\rho(u)\leq b+\varepsilon.
\]
Otherwise nearby exact critical points would converge by \eqref{eq:familycompact} to an excluded one. Set $J=J_\rho$, $F=J_{\rho-\delta}$ and apply Lemma~\ref{lem:commonfield} on $D=\{a-\varepsilon<J<b+\varepsilon\}$. Starting at $a<J(u_0)\leq b$, solve $\dot u=W(u)$ until $J=a$. Both $J,F$ decrease, and
\begin{equation}\label{eq:normidentity}
Q=\frac\rho\delta F-\Big(\frac\rho\delta-1\Big)J,
\quad  Q(u(t))\leq\frac\rho\delta F(u_0)-\Big(\frac\rho\delta-1\Big)a.
\end{equation}
Each trajectory is bounded, locally uniformly in $u_0$. Bounded speed rules out finite-time failure before contact: a finite-time endpoint is Cauchy in $E$, has $J\in[a,b]$, and belongs to $D$, where the local solution continues. On the bounded part of the closed strip visited by the orbit, Lemma~\ref{lem:commonfield} gives $m\geq\eta>0$. Equation \eqref{eq:commonfield} therefore gives $\dd J/\dd t\leq-\eta/4$, so the first hitting time $T(u_0)$ is finite.

Continuous dependence and strict decrease at contact give continuity of $T$ away from $J=a$: extend the local flow slightly beyond contact and use two nearby times on opposite sides of the level. Near $u_*\in\{J=a\}$, the norm bound in \eqref{eq:normidentity} and the corresponding $\eta$ are uniform, so
\[
0\leq T(u_0)\leq4(J(u_0)-a)/\eta\to 0.
\]
Set $T=0$ on $J^a$. If $\phi$ is the flow, $H(s,u)=\phi(sT(u),u)$ on the moving part and $H(s,u)=u$ on $J^a$ is the required strong deformation retraction. Decrease of $J$ preserves all its sublevels.
\end{proof}

The whole closed strip must be free of critical points; regularity of its boundary levels alone is insufficient. In particular, Proposition~\ref{prop:strip} applies between any two sufficiently negative levels.

\begin{proposition}\label{prop:high}
For every nonresonant $\rho>0$, there exists $b_*$ such that $J_\rho^b$ is a strong deformation retract of $E$ for every $b\geq b_*$. It is contractible.
\end{proposition}
\begin{proof}
Choose $\delta>0$ with $[\rho-\delta,\rho]$ nonresonant and $b_*>J_\rho(0)+1$ so that every family critical point has $J_\rho<b_*-2$. For $b\geq b_*$ use the common unit descent field on $D=\{J_\rho>b-1\}$ until $J_\rho=b$. Formula \eqref{eq:normidentity}, with lower bound $b$, bounds each trajectory. Lemma~\ref{lem:commonfield}, applied to the closed set $\{J_\rho\geq b\}\subset D$, gives a positive lower bound for $m$ on the bounded part visited by that trajectory, locally uniformly in its initial point. The stopping-time argument of Proposition~\ref{prop:strip} therefore gives a strong deformation retraction onto $J_\rho^b$. If $r_b:E\to J_\rho^b$ is the resulting retraction, $(s,u)\mapsto r_b((1-s)u)$ contracts $J_\rho^b$ to zero.
\end{proof}

\begin{proposition}\label{prop:margin}
Let $\rho>0$ be nonresonant, and choose $\vartheta>1$ such that $\rho/\vartheta$ and $\rho$ belong to the same nonresonant component. For every sufficiently negative $a$,
\[
U_1=\{J_\rho+(\vartheta-1)Q\leq a\}=J_{\rho/\vartheta}^{a/\vartheta}
\]
is a strong deformation retract of $U_0=J_\rho^a$. The deformation preserves every sublevel of $J_\rho$.
\end{proposition}
\begin{proof}
Choose $0<\lambda_0<\lambda_1=\rho/\vartheta$ in the same component, and set $I=[\lambda_0,\rho]$, $F=J_{\lambda_0}$, $G=\vartheta J_{\lambda_1}$. By \eqref{eq:familycompact}, choose $a_*<0$ such that all family critical points have $J_\rho>a_*+2$. For $a\leq a_*$ use the common field on $D=\{J_\rho<a+1\}$. On each bounded part of the closed set $\{J_\rho\leq a\}$, write $m\geq\eta>0$ for a positive lower bound. Both $J_\rho,F$ decrease, and $J_{\lambda_1}=\zeta J_\rho+(1-\zeta)F$ for $\zeta=(\lambda_1-\lambda_0)/(\rho-\lambda_0)\in(0,1)$, so $G$ decreases with $\dd G/\dd t\leq-\vartheta m/4$. Before its first contact with $a$,
\begin{equation}\label{eq:marginidentity}
Q(u)=\frac{F(u)-(\lambda_0/\rho)G(u)}{1-\vartheta\lambda_0/\rho}
\leq\frac{F(u_0)-(\lambda_0/\rho)a}{1-\vartheta\lambda_0/\rho}.
\end{equation}
Since $\lambda_0<\rho/\vartheta$, the factor $1-\vartheta\lambda_0/\rho$ in \eqref{eq:marginidentity} is positive. Thus the orbit remains bounded and inside $U_0$. Bounded speed gives continuation and the positive lower bound for $m$ gives finite contact. The bound is locally uniform in $u_0$; continuous dependence and strict crossing make the hitting time continuous away from $U_1$, and $T(u)\leq4(G(u)-a)/(\vartheta\eta)$ gives continuity at $U_1$ when $T$ is extended by zero. The stopped flow supplies the stated deformation.
\end{proof}

The additional term $(\vartheta-1)Q$ bounds the error $\varepsilon Q$ in estimates $J_\rho(H_tu)\leq J_\rho(u)+\varepsilon Q(u)+C$ with $\varepsilon<\vartheta-1$; the constant $C$ is accounted for by decreasing the sufficiently negative level, using Proposition~\ref{prop:strip}. Equivalently, the common field is transverse to $\{J_\rho+s(\vartheta-1)Q=a\}$ for $0\leq s\leq1$. The fixed lower parameter $\lambda_0$ supplies \eqref{eq:marginidentity}; neither differentiability of $A^{-1}$ nor a compactness assumption on unbounded approximate critical sequences has entered.

\subsection{Finite-dimensional reduction}
In the degree comparison below, set $\Gcal=\rho B'$ and $K_\rho=A^{-1}\Gcal$; the calligraphic symbol denotes an $E^*$-valued map, unlike the scalar $B=\ln  M$.
\begin{lemma}\label{lem:Galerkin}
Let $U\subset E$ be bounded and open with $u\ne K_\rho(u)$ on $\partial U$. Choose increasing finite-dimensional spaces $E_n\subset E$ with dense union and any inner product $(\cdot,\cdot)_n$ on each, with norm $\norm{v}_n=(v,v)_n^{1/2}$. Define $A_n:E_n\to E_n$ and $\Gcal_n:E\to E_n$ by
\[
(A_n(v),h)_n=\pair{A(v)}h,\quad (\Gcal_n(u),h)_n=\pair{\Gcal(u)}h\quad(h\in E_n).
\]
For all sufficiently large $n$,
\[
\degLS(\Id-K_\rho,U,0)=\degB(A_n-\Gcal_n,U\cap E_n,0).
\]
\end{lemma}
\begin{proof}
  Its proof, including the sign, is in Appendix~\ref{app:technical}.  
\end{proof}

Fix $0<\delta<\rho$ with $[\rho-\delta,\rho]$ nonresonant, set $J=J_\rho$, $F=J_{\rho-\delta}$, and let $\mathcal C$ be the compact union of their family critical sets. Choose
\[
C>2+\sup_{u\in\mathcal C}F(u),\quad  \Mcap=\{F\leq C\},
\]
with any such $C$ if $\mathcal C$ is empty. The notation $\Mcap$ always denotes this constraint set, while $M(u)$ remains the mass integral.

\begin{lemma}\label{lem:localize}
For $a_0<a_1<b$, let $Z=J^{a_0}\cup(J^b\cap\Mcap)$ and $V=\{J<a_1\}$. There exists a strong deformation retraction of $J^b$ onto $Z$, preserving all sublevels of $J$. Consequently, over any fixed coefficient field,
\[
H_j(J^b,V)\cong H_j(J^b\cap\Mcap,V\cap\Mcap)\quad(j\geq0).
\]
For finite $a<b$, $\Mcap\cap\{a\leq J\leq b\}$ is bounded in $E$.
\end{lemma}
\begin{proof}
On $D=\{F>C-1\}$ the family has no critical points. Use Lemma~\ref{lem:commonfield} to decrease both $J,F$. Starting in $J^b\setminus Z$, stop when either $J=a_0$ or $F=C$. Before contact,
\[
Q(u(t))\leq(\rho/\delta)F(u_0)-(\rho/\delta-1)a_0.
\]
The orbit is bounded. Bounded speed prevents finite-time failure before contact, and the positive lower bound for $m$ on the bounded part of $\{F\geq C\}$ makes contact finite. At either stopping level the corresponding function has strictly negative derivative; at simultaneous contact the stopping time is the minimum of the two local crossing times. These times, and hence their minimum, vary continuously. Locally uniform decrease also makes the stopped time tend to zero at $Z$. Thus the stopped flow is the claimed retraction and an equivalence of pairs $(J^b,V)\simeq(Z,Z\cap V)$.

Set $D_0=J^b\cap\Mcap$. The closure in $Z$ of $Z\setminus D_0$ is contained in $J^{a_0}$, in the relative interior of $Z\cap V$ because $a_0<a_1$. Excision gives $H_j(Z,Z\cap V)\cong H_j(D_0,D_0\cap V)$. Finally, \eqref{eq:normidentity} bounds $Q$ whenever $F\leq C$ and $J\geq a$.
\end{proof}

\begin{lemma}\label{lem:capstrip}
If $a<b$ and $J$ has no critical point in $\Mcap\cap\{a\leq J\leq b\}$, then $\Mcap\cap J^a$ is a strong deformation retract of $\Mcap\cap J^b$, by a deformation preserving every sublevel of $J$ in $\Mcap$.
\end{lemma}
\begin{proof}
  The proof in Appendix~\ref{app:technical} uses common descent near $F=C$ to keep trajectories inside the constraint.  
\end{proof}

For the next lemma choose $a<b$ and $C$ with the stricter inequalities
\begin{equation}\label{eq:facegap}
a+3<J(u)<b-3,\quad  F(u)<C-3
\quad\text{whenever }J_\lambda'(u)=0,\ \rho-\delta\leq\lambda\leq\rho.
\end{equation}
These choices are possible by \eqref{eq:familycompact}.
\begin{lemma}\label{lem:penalty}
The relative homology $H_*(\Mcap\cap J^b,\Mcap\cap J^a;\Q)$ has finite total rank. For all sufficiently large $n$,
\begin{equation}\label{eq:homologyreduction}
H_*(\Mcap\cap J^b,\Mcap\cap J^a;\Q)
\cong H_*(E_n\cap\Mcap\cap J^b,E_n\cap\Mcap\cap J^a;\Q).
\end{equation}
The open set $U=\{a<J<b,F<C\}$ is bounded and
\begin{equation}\label{eq:EulerBrouwer}
\chi(\Mcap\cap J^b,\Mcap\cap J^a)=\degB(\nabla_n(J|_{E_n}),U\cap E_n,0),
\end{equation}
where $\nabla_n$ is the gradient for $(\cdot,\cdot)_n$.
\end{lemma}
\begin{proof}
    The proof is in Appendix~\ref{app:technical}. It adds the same distance term to $J$ and $F$ and uses ordinary quotient topology; no local triviality of a family of quotients is assumed.
\end{proof}

\subsection{Degree--Euler identity}
\begin{theorem}\label{thm:degreeEuler}
For nonresonant $\rho>0$, all sufficiently negative $a$, and all sufficiently positive $b$, the relative rational homology has finite total rank and
\[
d_\rho=\chi(J_\rho^b,J_\rho^a)=1-\chi(J_\rho^a).
\]
No isolation or nondegeneracy of the solutions is required.
\end{theorem}
\begin{proof}
Choose $a,b,C$ satisfying \eqref{eq:facegap}. Take $a_1>a$ below all critical values of $J$, and apply Lemma~\ref{lem:localize} with $a_0<a$ and lower open member $\{J<a_1\}$. Propositions~\ref{prop:strip} and Lemma~\ref{lem:capstrip} give
\[
H_*(\{J<a_1\},J^a)=0,\quad 
H_*(\Mcap\cap\{J<a_1\},\Mcap\cap J^a)=0.
\]
Indeed each finite chain has compact image in a closed sublevel $J^s$, $a<s<a_1$, where the corresponding deformation applies. The direct limit over such $s$ and the exact sequences of triples give
\[
H_*(J^b,J^a)\cong H_*(\Mcap\cap J^b,\Mcap\cap J^a).
\]
Lemma~\ref{lem:penalty} proves finite total rank and computes the relative Euler characteristic by \eqref{eq:EulerBrouwer}. Identity \eqref{eq:normidentity} makes $U$ bounded, and \eqref{eq:facegap} places every critical point in its interior. Lemma~\ref{lem:Galerkin} and excision identify this degree with $d_\rho$. By Proposition~\ref{prop:high}, $J^b$ is contractible. Its pair exact sequence gives finite total rank for $H_*(J^a)$ and $\chi(J^b,J^a)=1-\chi(J^a)$. The same statement holds for the empty lower sublevel with $\chi(\varnothing)=0$.
\end{proof}

\section{Low sublevels and weighted barycenters}\label{sec:topology}
To apply Theorem~\ref{thm:degreeEuler}, we now identify the homotopy type of the negative sublevels. Set $J=J_{\rho,H}$ and write
\begin{equation}\label{eq:weightedenergy}
\Fcal_a(u)=(1+a)Q(u)-\rho\ln  M(u),\quad \Fcal_0=J_{\rho,H}.
\end{equation}
The definition \eqref{eq:weightedenergy} distinguishes the additional energy term $aQ$ from the nearby-parameter constraint $F$ of Section~\ref{sec:deformation}. H\"older's inequality and \eqref{eq:expbounded}, applied to $|u|$, show that
\[
E_{\mathrm{abs}}(u):=\int_\Om He^{|u|}\dd x
\]
is finite and continuous on $E$. For a finitely supported measure $\sigma$, denote its weighted support cardinality by $\card_\beta(\sigma)$.

\subsection{Concentrating test functions}\label{subsec:test}
The position-dependent exponent below adapts the regular-to-singular interpolation of Carlotto--Malchiodi \cite[Section 4, equations (29)--(30)]{CM}. Here the profiles have zero Dirichlet trace and are normalized individually before being combined. This gives the exact total mass and the prescribed limiting barycentric coefficients in Proposition~\ref{prop:test}, with the $N$-dimensional energy cost $c_N(1+\alpha_i)^{N-1}$.

Fix $K\Subset\Om$ containing $P$, $r=\rho/c_N$, and suppose $\Br(K)$ is nonempty and $c_Nr\notin\Sigma$. The finite set of costs below $r$ has a largest member, so
\[
\delta_r=r-\max_{\sigma\in\Br(K)}\card_\beta(\sigma)>0.
\]
\begin{proposition}\label{prop:test}
For every sufficiently large $\lambda$, there exists a continuous map $\Phi_\lambda:\Br(K)\to E$ with
\[
\sup_{\sigma\in\Br(K)} J_{\rho,H}(\Phi_\lambda(\sigma))\to -\infty,\quad 
\sup_{\sigma\in\Br(K)} d_{\mathrm{BL}}(\mu_{\Phi_\lambda(\sigma)},\sigma)\to 0.
\]
Continuity includes collisions, vanishing coefficients, and approaches of ordinary points to points of $P$. Moreover
\[
M(\Phi_\lambda(\sigma))=M(0)+\lambda^q,\quad 
J_{\rho,H}(\Phi_\lambda(\sigma))\leq q(c_N\card_\beta(\sigma)-\rho)\ln \lambda+C
\]
uniformly in $\sigma$.
\end{proposition}
\begin{proof}
Set $\alpha_*=\min(\{0\}\cup\{\alpha_i:1\leq i\leq m\})$. Choose $R>0$ so $B_R(y)\Subset\Om$ for $y\in K$, and $0<\delta<1$ so the balls $B_{3\delta}(p_i)$ are disjoint. For $\lambda>1$ define a continuous exponent $a_\lambda:K\to[\alpha_*,0]$ by zero outside $\bigcup_{i=1}^m B_\delta(p_i)$ and, when $0<d=|y-p_i|<\delta$, by
\[
T=\frac{\ln (\delta/d)}{\ln \lambda},\quad 
a_\lambda(y)=\max\Big\{\alpha_i,\frac{\alpha_iT}{1-\alpha_iT}\Big\},\quad 
a_\lambda(p_i)=\alpha_i.
\]
Thus $a_\lambda=\alpha_i$ for $d\leq\delta\lambda^{-1/(1+\alpha_i)}$; in the intermediate region, with $a=a_\lambda(y)$,
\begin{equation}\label{eq:transition}
\lambda^{a/(1+a)}(d/\delta)^{\alpha_i}=1.
\end{equation}
For $\alpha_i=0$ the exponent remains zero.

Set $\zeta=q(1+a)>0$ and
\[
P_{\lambda,y}(x)=
\left\{\begin{aligned}
    &N\ln \dfrac{1+\lambda^qR^\zeta}{1+\lambda^q|x-y|^\zeta}&&\text{ if }\,|x-y|<R,\\&0&&\text{ if }\,|x-y|\geq R,
\end{aligned}\right.
\quad  Z_{\lambda,y}=\int_\Om H(e^{P_{\lambda,y}}-1)\dd x.
\]
The profile has zero trace on its outer sphere and belongs to $E$. Uniformly in $y$,
\begin{equation}\label{eq:Zlower}
Z_{\lambda,y}\geq c\lambda^q.
\end{equation}
Indeed, set $\varepsilon=\lambda^{-1/(1+a)}$. On $B_\varepsilon(y)$, $e^{P_{\lambda,y}}-1\geq c\lambda^{qN}$. Away from points with negative exponents, $a=0$, $H\geq c$, and $\varepsilon=\lambda^{-1}$, which gives the claim. Near a point with negative exponent,
\[
\int_{B_\varepsilon(y)}H\dd x\geq c\varepsilon^N(d+\varepsilon)^{N\alpha_i}.
\]
If $a=\alpha_i$, then $d\leq\delta\varepsilon$, and the lower bound is $c\lambda^{qN}\varepsilon^{N(1+\alpha_i)}=c\lambda^q$. Otherwise \eqref{eq:transition} gives $\varepsilon/d\leq1/\delta$ and $\lambda^{Na/(1+a)}d^{N\alpha_i}=\delta^{N\alpha_i}$, again yielding \eqref{eq:Zlower}. No upper bound for $Z_{\lambda,y}/\lambda^q$ is required.

For $\sigma=\sum_{i=1}^k t_i\delta_{y_i}$, where $k\in\N$, the $y_i$ are distinct, $t_i>0$, and $\sum_{i=1}^k t_i=1$, define
\begin{equation}\label{eq:testmap}
\Phi_\lambda(\sigma)=\ln \Big(1+\lambda^q\sum_{i=1}^k t_i\frac{e^{P_{\lambda,y_i}}-1}{Z_{\lambda,y_i}}\Big).
\end{equation}
The sum is an integral against $\sigma$, hence is independent of its representation. Exponentiating \eqref{eq:testmap} and integrating against $H\dd x$ gives the asserted mass identity.

For fixed $\lambda$, the map $y\mapsto P_{\lambda,y}$ is continuous into $C^0(\bOm)\cap E$. Translation continuity and continuity of the exponent reduce the energy assertion to centered radial profiles. Their derivatives are bounded near zero by $C_\lambda s^{\zeta_{\min}-1}$, where $\zeta_{\min}=q(1+\alpha_*)>0$; its $N$th power is radially integrable. Dominated convergence proves the claim. Thus $y\mapsto(e^{P_{\lambda,y}}-1)/Z_{\lambda,y}$ is continuous into the same Banach space. A continuous Banach-valued function on compact $K$ can be uniformly approximated by finite sums of scalar continuous functions times fixed vectors, using a finite partition of unity. Integration against weakly convergent measures is therefore strongly continuous. The argument of the logarithm in \eqref{eq:testmap} is at least one and uniformly bounded for fixed $\lambda$; composition by logarithm is continuous in $C^0\cap E$. This includes all the stated limiting configurations.

Let $\nu_{\lambda,y}=H(e^{P_{\lambda,y}}-1)\dd x/Z_{\lambda,y}$. For fixed $s>0$, $P_{\lambda,y}$ is uniformly bounded on $|x-y|\geq s$, and \eqref{eq:Zlower} gives
\[
\sup_{y\in K}\nu_{\lambda,y}(\{|x-y|\geq s\})\leq C_s\lambda^{-q}.
\]
Hence $\nu_{\lambda,y}\to\delta_y$ uniformly in $d_{\mathrm{BL}}$. Since
\[
\mu_{\Phi_\lambda(\sigma)}=\frac{H\dd x+\lambda^q\sum_{i=1}^k t_i\nu_{\lambda,y_i}}{M(0)+\lambda^q},
\]
concentration is uniform in $\sigma$.

The same leading energy coefficient is obtained after normalization. Radial substitution gives
\[
Q(P_{\lambda,y})=c_N(1+a)^{N-1}\int_0^{\lambda^qR^\zeta}\frac{t^{N-1}}{(1+t)^N}\dd t
=qc_N(1+a)^{N-1}\ln \lambda+O(1),
\]
uniformly in $a\in[\alpha_*,0]$. Set $A=\lambda^q/Z_{\lambda,y}\leq C$ and $U_{\lambda,y}=\ln (1+A(e^{P_{\lambda,y}}-1))$. For $A\leq1$, $|\nabla U|\leq|\nabla P|$. For $A\geq1$, the derivative multiplier is at most $1+(C-1)e^{-P}$. Direct radial integration gives
\[
\int_\Om e^{-P_{\lambda,y}}|\nabla P_{\lambda,y}|^N\dd x
=S_N(N\zeta)^{N-1}\Big(\frac{\lambda^qR^\zeta}{1+\lambda^qR^\zeta}\Big)^N\leq C.
\]
It follows that $Q(U_{\lambda,y})\leq qc_N(1+a_\lambda(y))^{N-1}\ln \lambda+C$. Since $e^{\Phi_\lambda(\sigma)}=\sum_{i=1}^k t_ie^{U_{\lambda,y_i}}$, the logarithmic gradient is a convex combination of $\nabla U_{\lambda,y_i}$. Convexity and the uniform bound on the number of support points give
\[
Q(\Phi_\lambda(\sigma))\leq\sum_{i=1}^kQ(U_{\lambda,y_i})\leq qc_N\card_\beta(\sigma)\ln \lambda+C.
\]
At an ordinary point $(1+a_\lambda)^{N-1}\leq1$, and at $p_i$ it equals $\beta_i$. The mass identity obtained from \eqref{eq:testmap} and $\delta_r>0$ finish the proof.
\end{proof}

\subsection{Local deformations}
On a ball $B_R(c)$ we assume
\begin{equation}\label{eq:ballweight}
H(x)=h_c(x)|x-c|^{N\widehat\alpha},\quad  -1<\widehat\alpha\leq0,\quad 
\frac{\inf_{B_R(c)}h_c}{\sup_{B_R(c)}h_c}\geq\kappa>0.
\end{equation}
A ball centered at a point of $P$ is centered at one point of $P$, with its exponent. An ordinary ball has $\widehat\alpha=0$. One common $\kappa$ works for ordinary balls whose doubles avoid $P$ and sufficiently small balls centered at points of $P$ avoiding other points of $P$, since each remaining distance factor varies by a bounded ratio.

\begin{lemma}\label{lem:radialize}
For every $\varepsilon>0$, there exists $A_\varepsilon>1$ such that, if a ball satisfying \eqref{eq:ballweight} and $0<r_{\mathrm{in}}<R$ obey $R/r_{\mathrm{in}}\geq A_\varepsilon$, there exists a bounded linear homotopy $\mathcal A_t:E\to E$ preserving $u$ on $B_{r_{\mathrm{in}}}(c)$ and outside $B_R(c)$, with $\mathcal A_0u=u$ and $\mathcal A_1u$ radial on a fixed intermediate annulus. It is jointly continuous and satisfies
\[
Q(\mathcal A_tu)\leq(1+\varepsilon)Q(u),\quad  M(\mathcal A_tu)\leq\kappa^{-1}M(u).
\]
For $0<\vartheta\leq1$, if the inner ball contains at least $\vartheta M(u)$ of the weighted mass, then $M(\mathcal A_tu)\geq\vartheta M(u)$ and
\[
J(\mathcal A_tu)\leq J(u)+\varepsilon Q(u)-\rho\ln \vartheta.
\]
For finitely many disjoint outer balls, the same assertions hold with their preserved inner masses added.
\end{lemma}
\begin{proof}
    See Appendix~\ref{app:technical} for the proof.
\end{proof}

\begin{lemma}\label{lem:dilate}
Let $k\in\N$, and let $B_i=B_{R_i}(c_i)\Subset\Om$, $1\leq i\leq k$, have disjoint closures and satisfy \eqref{eq:ballweight}, with exponents $\widehat\alpha_i$ and a common $\kappa$. Set $\theta_i=1+\widehat\alpha_i$, $\omega_i=\theta_i^{N-1}$ and $W=\sum_{i=1}^k\omega_i$. Suppose $w\in E$ has constant trace $b_i$ on each $\partial B_i$. For $s\geq1$, $s_i=s^{1/\theta_i}$, set $\mathcal C_sw=w$ outside the balls and
\[
\mathcal C_sw(x)=\left\{\begin{aligned}
&w(c_i+s_i(x-c_i))+\gamma_N\ln  s&&\text{ if }\, |x-c_i|<R_i/s_i,\\
&b_i+\gamma_N\theta_i\ln (R_i/|x-c_i|)&&\text{ if }\, R_i/s_i\leq|x-c_i|<R_i.
\end{aligned}\right.
\]
Then $\mathcal C_sw\in E$, jointly continuously in $(s,w)$ on the subspace with constant spherical traces, and
\begin{equation}\label{eq:dilation}
Q(\mathcal C_sw)=Q(w)+qc_NW\ln  s,\quad 
\int_{\bigcup_{i=1}^k B_{R_i/s_i}(c_i)}He^{\mathcal C_sw}\dd x\geq\kappa s^q\int_{\bigcup_{i=1}^k B_i}He^w\dd x.
\end{equation}
Replacing $w,b_i$ by $(1-t)w,(1-t)b_i$ in the formula, retaining the logarithmic terms, and using $(1-t)w$ outside, defines $\mathcal D_{s,t}w$ with
\[
Q(\mathcal D_{s,t}w)=(1-t)^NQ(w)+qc_NW\ln  s,\quad
\mathcal D_{s,0}w=\mathcal C_sw,\quad \mathcal D_{s,1}w=P_s.
\]
The profile $P_s$ is independent of $w$ and supported in the outer union.
\end{lemma}
\begin{proof}
   The proof is in Appendix~\ref{app:technical}. 
\end{proof}

\begin{lemma}\label{lem:contract}
Fix $k\in\N$ concentric pairs $B_i^{\mathrm{in}}\Subset B_i^{\mathrm{out}}\Subset\Om$, $1\leq i\leq k$, with disjoint outer closures and common $\kappa$ in \eqref{eq:ballweight}, and total cost $W$. Suppose
\[
0\leq b<a,\quad  (1+b)c_NW<\rho,
\]
and the annular ratios are large enough to apply Lemma~\ref{lem:radialize} with $(1+b)(1+\varepsilon)\leq1+a$. Given $0<\eta'<1/2$, choose
\[
0<\eta<\min\{1/2,\kappa\eta'/4\},\quad  C=-\rho\ln [\kappa(1-\eta)].
\]
For every $L>0$, define
\begin{align*}
\mathscr A&=\Big\{u\in E:\Fcal_a(u)\leq-L-C,\int_{\bigcup_{i=1}^k B_i^{\mathrm{in}}}He^u\dd x\geq(1-\eta)M(u)\Big\},\\
\mathscr C&=\Big\{u\in E:\Fcal_b(u)\leq-L,\int_{\bigcup_{i=1}^k B_i^{\mathrm{out}}}He^u\dd x\geq(1-\eta')M(u)\Big\}.
\end{align*}
The inclusion $\mathscr A\hookrightarrow\mathscr C$ is homotopic in $\mathscr C$ to a constant map. The homotopy is jointly continuous and ends at one fixed truncated logarithmic profile. The allowable $\eta$ and additive loss $C$ are independent of the radii and centers.
\end{lemma}
\begin{proof}
Write $I=\bigcup_{i=1}^k B_i^{\mathrm{in}}$ and $O=\bigcup_{i=1}^k B_i^{\mathrm{out}}$. Simultaneous radialization gives
\[
\Fcal_b(\mathcal A_tu)\leq\Fcal_a(u)-\rho\ln (1-\eta)\leq-L.
\]
The exterior and inner functions remain unchanged. The exterior mass is at most $\eta M(u)$ and the inner mass at least $(1-\eta)M(u)$; their ratio to the total shows that the exterior fraction is at most $\eta<\eta'$. Let $w=\mathcal A_1u$ and choose one fixed sphere in each radial plateau, bounding a ball $\widehat B_i$ between the original inner and outer balls.

Apply Lemma~\ref{lem:dilate} on $\widehat B_i$. Set $\mathfrak b=q[\rho-(1+b)c_NW]>0$. The identities in \eqref{eq:dilation} give
\[
\Fcal_b(\mathcal C_sw)\leq\Fcal_a(u)+C-\mathfrak b\ln  s\leq-L-\mathfrak b\ln  s.
\]
Outside $O$ the mass remains at most $\eta M(u)$, whereas the rescaled inner balls contain at least $\kappa s^q(1-\eta)M(u)$. Thus the exterior fraction is at most $2\eta/(\kappa s^q)<\eta'$.

To remove the amplitude, use $\mathcal D_{e^T,t}w$. Weighted dilation in any one fixed ball followed by Jensen's inequality for the exponential function with respect to the normalized measure $H(x)\dd x$ on that ball gives geometric constants $c_0,c_1>0$ such that
\begin{equation}\label{eq:corelower}
\int_O He^{\mathcal D_{e^T,t}w}\dd x\geq c_0e^{qT-c_1\norm w_E}\quad(T\geq0,\ 0\leq t\leq1).
\end{equation}
Indeed the core mass is at least $\kappa e^{qT}\int_{\widehat B_i}He^{(1-t)w}\dd x$, and the average of $w$ for the probability measure $H(x)\dd x/\int_{\widehat B_i}H(y)\dd y$ has absolute value at most $C\norm w_E$, by H\"older's inequality and a finite-exponent Sobolev embedding. Outside $O$,
\begin{gather*}
\int_{\Om\setminus O}He^{\mathcal D_{e^T,t}w}\dd x\leq E_{\mathrm{abs}}(w),\\
\Fcal_b(\mathcal D_{e^T,t}w)\leq(1+b)Q(w)+\rho c_1\norm w_E-\rho\ln  c_0-\mathfrak bT.
\end{gather*}
Define continuous nonnegative scales
\begin{align*}
T_{\mathrm{mass}}(w)&=\frac{c_1\norm w_E+\ln ^+\!\Big(\frac{(1-\eta')E_{\mathrm{abs}}(w)}{\eta'c_0}\Big)}q,\\
T_{\mathrm{energy}}(w)&=\frac{[(1+b)Q(w)+\rho c_1\norm w_E-\rho\ln  c_0+L]^+}{\mathfrak b}.
\end{align*}
For $T$ above both scales, \eqref{eq:corelower} and the preceding energy bound place the amplitude-removal path in $\mathscr C$. Choose $T_0\geq0$ with
\[
\mathfrak bT_0\geq L-\rho\ln  c_0,\quad 
c_0e^{qT_0}\geq\frac{1-\eta'}{\eta'}M(0).
\]
The same estimates at $w=0$ show $P_{e^T}\in\mathscr C$ for $T\geq T_0$. Set
\[
T(w)=T_0+\max\{T_{\mathrm{mass}}(w),T_{\mathrm{energy}}(w)\}.
\]
Concatenate the four paths
\[
\mathcal A_tu,\quad  \mathcal C_{\exp(tT(w))}w,\quad 
\mathcal D_{\exp(T(w)),t}w,\quad 
P_{\exp((1-t)T(w)+tT_0)},\quad 0\leq t\leq1.
\]
Consecutive paths agree at their joining endpoints, every point satisfies the two defining constraints, and the terminal profile is $P_{e^{T_0}}$, independent of $u$. Lemmas~\ref{lem:radialize} and \ref{lem:dilate}, together with continuity of $E_{\mathrm{abs}}$, give joint continuity. On a compact parameter family convergence to the common terminal value is uniform, which gives continuity of the cone extensions used below.
\end{proof}

\subsection{Weighted barycenters}\label{subsec:barycenters}
Let
\[
\mathscr A_r=\{(n,I):n\in\N_0, I\subset\{1,\ldots,m\}, n+|I|>0, n+\beta_I<r\}.
\]
When this finite set is nonempty, set
\[
R_*:=\max_{(n,I)\in\mathscr A_r}(n+\beta_I)<r,\quad 
M_*:=\max_{(n,I)\in\mathscr A_r}(n+|I|),
\]
and choose $a_*>0$ with $(1+a_*)R_*<r$. For every $(n,I)\in\mathscr A_r$, $n<r$ and $|I|\leq m$, so $M_*\leq\lfloor r\rfloor+m<\infty$. An admissible measure has exactly $n+|I|$ support points for its ordinary support count $n$ and singular support set $I$, and hence at most $M_*$. The empty case will be treated separately. The core $K$ is taken sufficiently close to $\partial\Om$ when the inward map below is fixed. Conditions involving the separation of distinct points of $P$ are omitted if $m\leq1$.

Here is a direct construction of the polyhedral core used throughout. In a tubular boundary neighborhood let $d$ be signed distance, positive inside, and extend it to a smooth function on a containing box, positive and bounded away from zero on the remaining interior and negative on the remaining exterior. Interpolate it affinely on a sufficiently fine shape-regular Euclidean simplicial mesh of size $\ell$. Taylor's formula on each simplex gives uniform convergence of the interpolant $d_\ell$ to $d$ and uniform convergence of its piecewise gradients to $\nabla d$ in a smaller collar. In the normal coordinates $(y,s)\mapsto y-s\nu(y)$, the function $s\mapsto d_\ell(y-s\nu(y))$ is therefore strictly increasing, with slopes at least $1/2$, wherever it is close to a fixed small positive value $\varepsilon$. There exists a unique continuous graph $s=s_\ell(y)>0$ on which $d_\ell=\varepsilon$; existence follows from the opposite signs at the ends of a narrow interval around $\varepsilon$, and strict monotonicity gives uniqueness and continuity. The set $K=\{d_\ell\geq\varepsilon\}$ is a finite union of convex polytopes, hence a finite Euclidean polyhedron, contained in $\Om$. It contains all points outside this thin collar, including a neighborhood of $P$. On each normal segment, the map $s\mapsto\max\{s,s_\ell(y)\}$, interpolated with the identity, gives a strong deformation retraction of $\bOm$ onto $K$, fixed on $K$. The construction can make the collar arbitrarily thin. A finite triangulation of the polytopes and subdivisions inserting the finitely many $p_i$ give the base triangulations below. Thus no semialgebraicity of the smooth domain is being assumed.

\begin{lemma}\label{lem:polyhedron}
The space $\Br(K)$ is compact and homeomorphic to a finite polyhedron. It has a continuous retraction from an open neighborhood in $\Prob(\bOm)$. For each $\ell>0$ it admits a finite triangulation $\mathscr T$ such that for every closed simplex $\tau$,
\[
S_\tau:=\bigcup_{\sigma\in\tau}\supp\sigma
\]
is covered by $n$ closed simplices of a triangulation of $K$ of mesh at most $\ell$, together with $\{p_i:i\in I\}$, for some $(n,I)\in\mathscr A_r$. If $\eta$ is a face of $\tau$, $S_\eta\subset S_\tau$. The target mesh in $d_{\mathrm{BL}}$ can also be arbitrarily small, and its dimension is at most $M_*(N+1)-1$.
\end{lemma}
\begin{proof}
Triangulate $K$ with mesh at most $\ell$ and each $p_i$ a vertex. For an admissible $(n,I)$ and ordered base simplices $Q_1,\ldots,Q_n$, form the compact image
\[
\mathscr C(n,I,Q)=\Big\{\sum_{j=1}^nt_j\delta_{x_j}+\sum_{i\in I}s_i\delta_{p_i}:
x_j\in Q_j, t_j,s_i\geq0, \sum_{j=1}^nt_j+\sum_{i\in I}s_i=1\Big\}.
\]
Ordinary slots cost at most one; fixed slots fixed at points of $P$ cost $\beta_i$. Identifications and zero coefficients cannot increase cost, so these images are contained in $\Br(K)$. They form a finite cover and prove compactness.

Let $\mathbf m$ be the vector of polynomial moments of degrees at most $2M_*$. It is injective on $\Br(K)$. For measures with the same moments, the nonnegative polynomial $\prod_{a\in\supp\mu}|x-a|^2$ has zero integral against both, so $\supp\nu\subset\supp\mu$. If the latter support is $\{a_1,\ldots,a_j\}$, the polynomials
\[
\ell_i(x)=\prod_{\substack{1\leq s\leq j\\s\ne i}}\frac{|x-a_s|^2}{|a_i-a_s|^2}
\]
identify its coefficients. Their degrees are at most $2M_*-2$. Compactness makes $\mathbf m$ a homeomorphism onto its image.

Each $\mathbf m(\mathscr C(n,I,Q))$ is a polynomial image of a product of finite simplices, hence is compact semialgebraic. We use simultaneous semialgebraic triangulation: a compact semialgebraic set and finitely many semialgebraic subsets admit a finite triangulation compatible with every subset \cite[Theorem 1.10, p.~9]{Coste}. All our specified subsets are closed, so compatibility makes them subcomplexes. Every open target simplex is contained in a member of the cover, and its closure stays there. This proves the support assertion. Face inclusion follows from the definition of $S_\tau$. Subdivision makes the mesh small in any compatible metric. The slot space has dimension
\[
nN+n+|I|-1\leq M_*(N+1)-1;
\]
semialgebraic images do not increase dimension (\cite[Proposition 1.8, p.~8]{Coste}), giving the uniform dimension bound.

Embed $\mathbf m(\Br(K))$ in the interior of a closed cube and triangulate the pair by the same theorem. After barycentric subdivision its subcomplex is full. On the open set where the sum of the barycentric coordinates at its vertices is positive, normalize just those coordinates. Their vertices span a simplex of the full subcomplex, so this defines a retraction. Conjugate by the triangulation homeomorphism and compose with the continuous moment map on $\Prob(\bOm)$. This gives
\[
\mathfrak r:\mathcal U\to \Br(K),\quad  \mathcal U\subset\Prob(\bOm)\text{ open}.
\]
If $\dist_{\mathrm{BL}}(\mu_j,\Br(K))\to0$, then $\mu_j\in\mathcal U$ for all sufficiently large $j$ and
\begin{equation}\label{eq:retractclose}
d_{\mathrm{BL}}(\mu_j,\mathfrak r(\mu_j))\to0.
\end{equation}
Choose nearby measures in the compact retract and extract a convergent subsequence; continuity of $\mathfrak r$ gives the same limit. The same compactness argument for every subsequence proves \eqref{eq:retractclose}.
\end{proof}

Choose a smooth inward flow $T_t$ supported in a boundary collar disjoint from $P$, so $T_t(\bOm)\Subset\Om$ for small $t>0$. Set
\[
(\mathcal T_tu)(x)=\left\{\begin{aligned}
&u(T_t^{-1}x),&&x\in T_t(\Om),\\&0,&&x\in\Om\setminus T_t(\Om).
\end{aligned}\right.
\]
Extension by zero is valid because the trace vanishes, and the resulting maps are jointly continuous in $E$: first verify it for smooth compactly supported functions and then use their density and the uniform bounds for composition by $T_t$. The positive function
\[
\mathfrak a_t(x)=\frac{H(T_tx)\det DT_t(x)}{H(x)}
\]
extends across $P$ with value one because $T_t$ is the identity there. In the collar $H$ is positive and $C^1$. Hence, uniformly for small $t$,
\[
e^{-Ct}\leq \mathfrak a_t\leq e^{Ct},\quad Q(\mathcal T_tu)\leq e^{Ct}Q(u),\quad M(\mathcal T_tu)\geq e^{-Ct}M(u).
\]
Choose $0<a_0<a_{-1}<a_*$ and fix $t_*>0$ so that $(1+a_0)e^{Ct_*}\leq1+a_{-1}$. Then
\begin{equation}\label{eq:pushmargin}
\Fcal_{a_0}(\mathcal T_tu)\leq\Fcal_{a_{-1}}(u)+C_{\mathrm{push}}\quad(0\leq t\leq t_*),\quad  C_{\mathrm{push}}=C\rho t_*.
\end{equation}
Choose $K$ close enough to the boundary that $T_{t_*}(\bOm)\subset\operatorname{int}K$, and fix the preceding retraction for this $K$. Write $T=T_{t_*}$ and $\mathcal T=\mathcal T_{t_*}$.

\begin{lemma}\label{lem:projection}
For all sufficiently large $L$,
\[
\Psi(u)=\mathfrak r(\mu_{\mathcal Tu}),\quad  u\in J_\rho^{-L},
\]
defines a continuous map to $\Br(K)$, with the same formula at every larger depth. Moreover
\begin{equation}\label{eq:projectionerror}
\varepsilon(D):=\sup_{J_\rho(u)\leq-D}d_{\mathrm{BL}}(\mu_{\mathcal Tu},\Psi(u))\to 0\quad(D\to\infty).
\end{equation}
For every sufficiently large $L$, there exists $\Phi:\Br(K)\to J_\rho^{-L}$ with $\Psi\Phi\simeq\Id$.
\end{lemma}
\begin{proof}
The probability-measure map
\[
\mathcal P_T(\mu)=\frac{T_\#(\mathfrak a_{t_*}\mu)}{\int_{\bOm} \mathfrak a_{t_*}(x)\dd\mu(x)}
\]
is continuous, hence uniformly continuous on compact $\Prob(\bOm)$. It sends $\Br(\bOm)$ into $\Br(K)$ because $T$ is injective and fixes $P$. By Lemma~\ref{lem:concentration}, $\mu_u$ approaches $\Br(\bOm)$ as $J_\rho(u)\to-\infty$. Also $J_\rho(u)\leq-D$ implies $M(u)\geq e^{D/\rho}$; the extra unnormalized mass of $\mathcal Tu$ outside $T(\Om)$ is at most $M(0)$. Consequently
\[
d_{\mathrm{BL}}(\mu_{\mathcal Tu},\mathcal P_T(\mu_u))\leq Ce^{-D/\rho}.
\]
Thus $\mu_{\mathcal Tu}$ approaches $\Br(K)$ uniformly, eventually belongs to $\mathcal U$, and \eqref{eq:retractclose} proves \eqref{eq:projectionerror}. Continuity follows from the weighted exponential integral and composition map.

Proposition~\ref{prop:test} gives $\Phi_\lambda$ into any prescribed negative sublevel for large $\lambda$, with $\mu_{\Phi_\lambda(\sigma)}\to\sigma$ uniformly. Therefore $\mu_{\mathcal T\Phi_\lambda(\sigma)}\to\mathcal P_T(\sigma)$ uniformly. For sufficiently large $\lambda$, straight interpolation of these measures stays in $\mathcal U$ by compactness of $\Br(K)$; applying $\mathfrak r$ gives $\Psi\Phi_\lambda\simeq\mathcal P_T|_{\Br(K)}$.

Let $P_K:\bOm\to K$ be the collar retraction. It fixes a neighborhood of $P$ and sends no ordinary point to a point of $P$. Then
\[
\sigma\mapsto\frac{(P_KT_t)_\#(\mathfrak a_t\sigma)}{\int_K \mathfrak a_t(x)\dd\sigma(x)},\quad 0\leq t\leq t_*,
\]
is a homotopy in $\Br(K)$ from the identity to $\mathcal P_T$: points of $P$ stay fixed, and ordinary collisions only lower cost. One sufficiently large $\lambda$ gives $\Phi=\Phi_\lambda$.
\end{proof}

\begin{lemma}\label{lem:carriers}
Fix $d_{\max}\in\N_0$ and finitely many annular preparation tolerances. For sufficiently small base mesh $\ell$ in Lemma~\ref{lem:polyhedron}, each $(\tau,d)$, $\tau\in\mathscr T$, $d\in\{0,\ldots,d_{\max}\}$, admits concentric inner and outer balls with disjoint outer closures in $\Om$. Each is centered at a single point of $P$ and contains no other point of $P$, or has center $c$ and outer radius $R$ with $\dist(c,P)>2R$. Their total cost is at most $R_*$. Their unions $I_{\tau,d}$ and $O_{\tau,d}$ satisfy
\begin{equation}\label{eq:carriernesting}
\dist(S_\tau,\Om\setminus I_{\tau,d})\geq \ell,\quad 
O_{\eta,e}\subset I_{\tau,d}\quad\text{if }\eta\leq\tau,\ e<d,
\end{equation}
including $\eta=\tau$. Their annular ratios are large enough for all prescribed applications of Lemma~\ref{lem:radialize}. All outer radii and distances of their points from $S_\tau$ are at most $C_*\ell$, with $C_*$ independent of $\ell$ and the number of target simplices. Their positive weight factors have one common comparison constant $\kappa\in(0,1]$, depending only on the fixed domain and weight, independently of the annular ratio, $d_{\max}$, and refinement.
\end{lemma}
\begin{proof}
Choose an annular ratio $A>1$ sufficient for the prescribed tolerances. Induct on $\dim\tau+d\leq M_*(N+1)-1+d_{\max}$, with increasing geometric bounds. Suppose the outer balls already chosen for $\eta\leq\tau$ and $e<d$ lie within $C_{\mathrm{prev}}\ell$ of $S_\eta$, and hence of $S_\tau$. Cover $S_\tau$ by the admissible slots of Lemma~\ref{lem:polyhedron}, omit unused slots, and choose a seed in each nonempty intersection of a base simplex with $S_\tau$; use $p_i$ for a slot fixed at a point of $P$. Identify coincident seeds. A seed at $p_i$ is assigned weight $\beta_i$, and every other seed weight one, so the total weight does not exceed the original slot weight. Inner balls of radius $(C_{\mathrm{prev}}+3)\ell$ cover $S_\tau$ and all these previously constructed outer balls so that the covered sets have distance at least $\ell$ from the complement of the union of these balls. Enlarge them by $A$.

If an ordinary outer ball $B_R(c)$ satisfies $|c-p_i|\leq2R$, replace it by an inner ball centered at $p_i$ of radius $3R+\ell$ and enlarge by $A$. Its cost changes from $1$ to $\beta_i\leq1$. If two outer closures meet, enclose their union in a new inner ball whose boundary is at distance at least $\ell$ from that union, and enlarge by $A$. Retain the center in $P$ if one ball is centered in $P$; otherwise choose an enclosing center. The new radius is bounded by a fixed multiple of the sum of the former radii plus $\ell$. Two ordinary balls become one; if one ball is centered at a point of $P$, the new ball retains its center and weight, so total cost cannot increase.

At most $M_*-1$ operations combine two balls into one. Recentering creates no ordinary ball and occurs at most $M_*$ times. If $V$ is the sum of outer radii, each operation gives $V_{\mathrm{new}}\leq C_A(V+\ell)$. A new bound $C_{\mathrm{next}}\ell$ therefore depends only on the preceding bound, $A$, and $M_*$. Each resulting ball contains an original seed, giving the same type of bound for its distance from $S_\tau$. There are uniformly finitely many ranks. Thus $C_*$ is fixed before $\ell$ is chosen, regardless of the number of simplices or previously constructed carriers.

Choose $\ell$ so the resulting diameter and distance bounds are smaller than a fixed fraction of the minimum separation of distinct points of $P$ and of $\dist(K,\partial\Om)$. At every intermediate operation balls centered at distinct points of $P$ cannot meet. Balls centered at the same point of $P$ are enclosed in one ball with that center under the preceding rule, and no ordinary ball can require recentering at two distinct points of $P$. The operations are therefore defined, preserve all containments and distance bounds, and keep cost at most the original $n+\beta_I\leq R_*$. For an ordinary ball whose double avoids $P$, each distance to $p_i$ varies by a factor at most three. For a ball centered at a point of $P$, the remaining factors stay uniformly bounded above and below. Together with the fixed positive bounds for $h$, these observations provide the common $\kappa$, with the asserted independence.
\end{proof}

\begin{theorem}\label{thm:reverse}
For every sufficiently large $L$ and nonresonant $\rho>0$, $J_\rho^{-L}\simeq\Br(K)$. If $\mathscr A_r=\varnothing$, all sufficiently negative sublevels are empty.
\end{theorem}
\begin{proof}
The empty case follows from Lemma~\ref{lem:concentration}. Otherwise choose $L_*>0$ so $\Psi$ is defined on $U_*=J_\rho^{-L_*}$ and all lower levels are regular. Keep the inward maps, retraction, core, projection, and parameters $a_{-1}>a_0>0$ fixed. Lemma~\ref{lem:projection} gives $\Psi\Phi\simeq\Id$. Since $r/(1+a_{-1})>R_*$ and every element of $\Sigma/c_N$ below $r$ is at most $R_*$, $\rho/(1+a_{-1})$ and $\rho$ belong to the same component; Proposition~\ref{prop:margin} applies.

We prove: if $f:S^n\to U_*$, $n\geq0$, and $\Psi f$ extends to $D^{n+1}$, then $f$ also extends there. Finite choices below may depend on $n$, but the fixed projection and its parameters do not. Choose
\[
a_0>a_1>\cdots>a_{n+1}\geq0.
\]
Choose $0<\eta_{n+1}<1/8$ and, for $d=n,\ldots,0$, $0<\eta_d<\eta_{d+1}/2$ small enough for Lemma~\ref{lem:contract} with mass tolerances $\eta_d$ and $\eta_{d+1}$. The choice uses the common $\kappa$, before choosing ball radii. Choose tolerances $\epsilon_d>0$ with $(1+a_d)(1+\epsilon_d)\leq1+a_{d-1}$, $1\leq d\leq n+1$. These fix the annular ratio and $C_*$ in Lemma~\ref{lem:carriers}.

Choose a sufficiently small base mesh $\ell<1$, triangulate the target by Lemma~\ref{lem:polyhedron}, and subdivide until its $d_{\mathrm{BL}}$ mesh is less than
\[
\delta=\eta_0\ell/16.
\]
The support covers of Lemma~\ref{lem:polyhedron} persist under subdivision. Construct the carriers for $0\leq d\leq n+1$. This geometry and target triangulation now stay fixed. For each $(\tau,d)$ with $d\geq1$, Lemma~\ref{lem:contract} with parameters $a_{d-1}>a_d$ and errors $\eta_{d-1},\eta_d$ has an additive loss $C_{\tau,d}$. Set $C_d=\max_{\tau\in\mathscr T}C_{\tau,d}$ and choose depths backwards from $L_{n+1}\geq L_*$ so that
\begin{equation}\label{eq:depths}
L_{d-1}\geq L_d+C_d+1\quad(d=n+1,\ldots,1).
\end{equation}
Increase them all by one common constant until $D=L_0+C_{\mathrm{push}}$ is in the range of Proposition~\ref{prop:margin} and $\varepsilon(D)<\delta$. All these choices precede subdivision of the parameter domain.

Propositions~\ref{prop:strip} and \ref{prop:margin} deform $f$ in $U_*$ to $\widehat f$ with $\Fcal_{a_{-1}}(\widehat f)\leq-D$. The fixed $\Psi$ is defined along this homotopy, so $\Psi\widehat f$ still extends to the ball. Set $w=\mathcal T\widehat f$. Equation \eqref{eq:pushmargin} gives a path in $U_*$ from $\widehat f$ to $w$, and
\[
\Fcal_{a_0}(w)\leq-L_0,\quad 
d_{\mathrm{BL}}(\mu_w,\Psi(\widehat f))<\delta.
\]
Choose $g:D^{n+1}\to\Br(K)$ extending $\Psi\widehat f$. Its boundary value is $\mathfrak r(\mu_w)$, rather than $\Psi(w)$.

Triangulate the pair $(D^{n+1},S^n)$ by a finite complex $\mathscr D$ with boundary subcomplex. Subdivide only this domain until $g$ has a simplicial approximation $g':\mathscr D\to\mathscr T$. By the open-star construction (\cite[Theorem 2C.1 and its open-star proof]{Hatcher}), $g(x)$ and $g'(x)$ belong to a common closed target simplex; hence
\begin{equation}\label{eq:simplicialclose}
d_{\mathrm{BL}}(g(x),g'(x))<\delta.
\end{equation}
We do not require equality of $g,g'$ on the boundary: the $E$-valued boundary map $w$ will remain unchanged. Let $\tau(\sigma)$ be the simplicial image of a domain simplex $\sigma$. Faces satisfy $\tau(\xi)\leq\tau(\sigma)$, possibly with equality. Assign to $\sigma$ the carrier $(\tau(\sigma),\dim\sigma)$.

For a target simplex $\tau$, set $\zeta_\tau(y)=\max\{1-\dist(y,S_\tau)/\ell,0\}$. It is one on $S_\tau$, zero outside $I_{\tau,d}$, and $\norm{\zeta_\tau}_\infty+\Lip(\zeta_\tau)\leq2/\ell$. On a closed boundary simplex $\sigma$, with $d=\dim\sigma$, $\tau=\tau(\sigma)$, the boundary identity for $g$ and \eqref{eq:simplicialclose} give
\begin{align}
    \mu_{w(x)}(I_{\tau,d})&\geq\int_\Om\zeta_\tau(y)\dd\mu_{w(x)}(y)\nonumber\\
&\geq\int_K\zeta_\tau(y)\dd[g'(x)](y)-\frac2\ell\bigl[d_{\mathrm{BL}}(\mu_{w(x)},g(x))+d_{\mathrm{BL}}(g(x),g'(x))\bigr]\nonumber\\
&>1-4\delta/\ell\geq1-\eta_0.\label{eq:boundarymass}
\end{align}
Estimate \eqref{eq:boundarymass} holds on every face of $\sigma$ as well.

Extend $w$ over the skeleton, requiring on each $d$-simplex $\sigma$ that
\begin{equation}\label{eq:skeletonbounds}
\Fcal_{a_d}(W(x))\leq-L_d,\quad 
\mu_{W(x)}(O_{\tau(\sigma),d})\geq1-\eta_d.
\end{equation}
The entire prescribed boundary already satisfies \eqref{eq:skeletonbounds} because $a_d\leq a_0$, $L_d\leq L_0$, and $\eta_d\geq\eta_0$. At an interior vertex use one fixed sufficiently concentrated logarithmic profile in its carrier. If its total cost is $W_{\tau,0}\leq R_*$, its energy has coefficient $q[(1+a_0)c_NW_{\tau,0}-\rho]<0$, and its exterior mass fraction tends to zero. Thus it satisfies \eqref{eq:skeletonbounds}. There are finitely many vertex carriers.

Suppose $W$ is defined on the boundary and the $(d-1)$-skeleton. For an interior $d$-simplex $\sigma$, set $\tau=\tau(\sigma)$. At $z\in\partial\sigma$, choose a proper face $\xi$ containing $z$, of dimension $e<d$. The induction, ordered parameters, and \eqref{eq:depths} give
\[
\Fcal_{a_{d-1}}(W(z))\leq\Fcal_{a_e}(W(z))\leq-L_e\leq-L_{d-1}\leq-L_d-C_{\tau,d}.
\]
Equation \eqref{eq:carriernesting}, including the case $\tau(\xi)=\tau$, gives
\[
\mu_{W(z)}(I_{\tau,d})\geq\mu_{W(z)}(O_{\tau(\xi),e})\geq1-\eta_e\geq1-\eta_{d-1}.
\]
Thus the entire boundary image belongs to the domain of the local contraction $H_{\tau,d}$, which ends at a fixed profile $P_{\tau,d}$ and satisfies the bounds in \eqref{eq:skeletonbounds} for $\sigma$. Let $b_\sigma$ be the geometric barycenter. For the unique expression $(1-t)z+tb_\sigma$ with $z\in\partial\sigma$, $0\leq t<1$, define
\[
W((1-t)z+tb_\sigma)=H_{\tau,d}(W(z),t),\quad  W(b_\sigma)=P_{\tau,d}.
\]
This agrees exactly on the boundary. Compactness of $W(\partial\sigma)$ and joint continuity give uniform convergence to $P_{\tau,d}$ as $t\to1$, proving continuity at the center. Common faces retain their previous definitions. Iterating to dimension $n+1$ gives $W:D^{n+1}\to U_*$ because $J_\rho\leq\Fcal_{a_d}\leq-L_d\leq-L_*$. Attach the earlier homotopy from $f$ to $w$ to this filling to fill $f$.

For $n=0$, the relative extension proves injectivity on path components; $\Psi\Phi\simeq\Id$ gives surjectivity. For $n\geq1$ it proves injectivity on based homotopy groups. Movement of the basepoint causes no difficulty: change of basepoint preserves the zero class, and for loops it is conjugation. The one-sided homotopy gives surjectivity after transporting the basepoint. Bijectivity on components allows transport to any prescribed basepoint. Hence $\Psi$ is a weak homotopy equivalence.

Finally, choose $\varepsilon>0$ so that $J_\rho$ has no critical point with value in $[-L_*,-L_*+\varepsilon]$. Proposition~\ref{prop:strip}, restricted to the open set $V=\{J_\rho<-L_*+\varepsilon\}$, makes $U_*$ a strong deformation retract of $V$. The Banach space $E$ is an absolute retract for metric spaces by \cite[Theorem 4.1 and Corollary 4.2, pp.~357--358]{Dugundji}. Thus the closed neighborhood retract $U_*\subset E$ is an absolute neighborhood retract (ANR); see \cite[Section 2, p.~391]{Hanner}. Since $U_*$ is separable and metrizable, \cite[Theorem 1, p.~272]{MilnorCW} gives CW homotopy type. The target is a finite polyhedron by Lemma~\ref{lem:polyhedron}. Whitehead's theorem \cite[Theorem 4.5]{Hatcher}, applied to CW models on each component, now makes $\Psi$ a homotopy equivalence. The argument starts at any sufficiently large $L_*$, proving the assertion at every such level.
\end{proof}

\section{Proofs of main results}\label{sec:formula}
Theorems~\ref{thm:degreeEuler} and \ref{thm:reverse} reduce the main results to the Euler characteristic of $\Br(K)$. We compute it using $\chi_c$ on the configuration strata of the compact core $K$ and verify the cancellation of zero exponents. We then complete the proofs of Theorems~\ref{thm:degree} and \ref{thm:homotopy} and Corollaries~\ref{cor:existence}, \ref{cor:onemark}, and \ref{cor:equalmarks}.

\begin{proposition}\label{prop:Euler}
If $c=\chi(K)$, $r>0$, and $c_Nr\notin\Sigma$, then
\begin{equation}\label{eq:baryEuler}
1-\chi(\Br(K))=\sum_{\substack{I\subset\{1,\ldots,m\}\\\beta_I<r}}(-1)^{|I|}
\binom{\lfloor r-\beta_I\rfloor-c+m}{\lfloor r-\beta_I\rfloor}.
\end{equation}
\end{proposition}
\begin{proof}
Set $X=K\setminus P$, with $\chi_c(X)=c-m$. A stratum is specified by $\ell$ ordinary support points and exactly the points of $P$ in $I$, all with positive coefficients, where $\ell+\beta_I<r$ and $\ell+|I|>0$. Writing $\operatorname{Conf}_\ell(X)=\{(x_1,\ldots,x_\ell)\in X^\ell:x_i\ne x_j\text{ for }i\ne j\}$, with $\operatorname{Conf}_0(X)$ a point, and denoting the standard $j$-simplex by $\Delta^j$, its ordered covering is
\[
\operatorname{Conf}_\ell(X)\times\operatorname{int}\Delta^{\ell+|I|-1},
\]
with the symmetric group on the $\ell$ ordinary points acting freely. All these spaces and actions are semialgebraic under the map $\mathbf m$. The additivity and multiplicativity of $\chi_c$ used here are the finite-cell properties of compactly supported Euler characteristic; equivalently see \cite[Lemma 1.21 and Theorem 1.22]{Coste}. For $z=\chi_c(X)$, inclusion--exclusion on the diagonals of $X^\ell$ yields
\[
\chi_c(\operatorname{Conf}_\ell(X))=z(z-1)\cdots(z-\ell+1).
\]
Indeed the diagonal for a partition into $j$ blocks is $X^j$, and the partition-lattice sum is the falling factorial. This uses no configuration-space fibration for a core with boundary. Triangulating a finite covering compatibly and counting the lifts of its open cells shows that its degree multiplies $\chi_c$. Also $\chi_c(\operatorname{int}\Delta^j)=(-1)^j$. Thus our stratum contributes
\[
(-1)^{\ell+|I|-1}\binom{c-m}{\ell}.
\]
No orientability is needed. Since $\Br(K)$ is compact, its Euler characteristic equals $\chi_c$. Additivity, with the empty stratum represented by the term one, gives
\[
1-\chi(\Br(K))=\sum_{\substack{I\subset\{1,\ldots,m\},\ \ell\geq0\\\ell+\beta_I<r}}
(-1)^{\ell+|I|}\binom{c-m}{\ell}.
\]
For $\beta_I<r$, nonresonance makes $r-\beta_I$ nonintegral, so the largest $\ell$ is $k_I=\lfloor r-\beta_I\rfloor$. The identity
\[
\sum_{\ell=0}^k(-1)^\ell\binom z\ell=(-1)^k\binom{z-1}k=\binom{k-z}k
\]
proves \eqref{eq:baryEuler}, also when the barycenter space is empty.
\end{proof}

Formula \eqref{eq:baryEuler} agrees with \cite[Theorem 1.1 and Corollary 1.2]{Kallel}. Its relation to the generating series is given by \cite[Corollary 1.8]{Kallel}. With $c=\chi(K)$ and locally finite real exponents, write
\begin{equation}\label{eq:series}
G(t)=(1-t)^{c-m}\prod_{i=1}^m(1-t^{\beta_i})=\sum_{\lambda\geq0}a_\lambda t^\lambda.
\end{equation}
For $r>0$ with $c_Nr\notin\Sigma$, the degree is $\sum_{0\leq\lambda<r}a_\lambda$, with $a_0=1$, strict cutoff, and repeated exponents combined. Let $0<r_-<r_0<r_+$ satisfy $c_Nr_\pm\notin\Sigma$, and suppose that $c_Nr_0$ is the only element of $\Sigma$ in $(c_Nr_-,c_Nr_+)$. Then $d_{c_Nr_+}-d_{c_Nr_-}=a_{r_0}$, possibly zero. Multiplication by $(1-t)^{-1}$ sums along integer translates; it does not give $\sum_{0\leq\lambda<r}a_\lambda$ for general real exponents.

\begin{proposition}\label{prop:zero}
Let $Z=\{i\in\{1,\ldots,m\}:\alpha_i=0\}$, $I_-=\{i\in\{1,\ldots,m\}:\alpha_i<0\}$, $m_-=|I_-|$, and $P_-=\{p_i:i\in I_-\}$. Deleting the points $\{p_i:i\in Z\}$ leaves $H,J_{\rho,H},K_{\rho,H},\Sigma$, and $\Br(K)$ unchanged. Formula \eqref{eq:degree} equals
\[
\sum_{\substack{I\subset I_-\\\beta_I<r}}(-1)^{|I|}
\binom{\lfloor r-\beta_I\rfloor-\chi(\Om)+m_-}{\lfloor r-\beta_I\rfloor}.
\]
If every exponent is zero, then $H=h$, $\Sigma=c_N\N$, and for $kc_N<\rho<(k+1)c_N$,
\[
d_\rho(\Om,h)=\binom{k-\chi(\Om)}k,\quad  k\in\N_0.
\]
At zero the degree is one.
\end{proposition}
\begin{proof}
For $i\in Z$ the factor is identically one and $\beta_i=1$, so the equation and functional are unchanged. Every resonance satisfies
\[
\ell+\beta_I=(\ell+|I\cap Z|)+\beta_{I\cap I_-},
\]
preserving positivity; conversely every reduced resonance occurs originally. For any finite support $S$,
\[
\#(S\setminus P)+\sum_{\substack{1\leq i\leq m\\p_i\in S}}\beta_i
=\#(S\setminus P_-)+\sum_{\substack{i\in I_-\\p_i\in S}}\beta_i.
\]
Thus the barycenter spaces are literally the same subsets of $\Prob(K)$. Finally \eqref{eq:series} has the exact cancellation
\[
(1-t)^{c-m}\prod_{i=1}^m(1-t^{\beta_i})
=(1-t)^{c-m}(1-t)^{|Z|}\prod_{i\in I_-}(1-t^{\beta_i})
=(1-t)^{c-m_-}\prod_{i\in I_-}(1-t^{\beta_i}).
\]
Their coefficients summed below the threshold, hence the finite sums in Proposition~\ref{prop:Euler}, agree. No limit in the exponents is involved.
\end{proof}

\begin{proof}[Proof of Theorems~\ref{thm:degree} and \ref{thm:homotopy}]
Theorem~\ref{thm:compact} makes the degree admissible on a sufficiently large ball of $E$. If $r<\beta_*$, replacing $\rho$ by $t\rho$, $0\leq t\leq1$, stays nonresonant and admits a common ball. At zero the map is the identity and the degree is one. Coercivity makes sufficiently negative sublevels empty. In every other nonresonant interval, Theorems~\ref{thm:degreeEuler} and \ref{thm:reverse} give \eqref{eq:mainhomotopy}. Apply Proposition~\ref{prop:Euler} and $\chi(K)=\chi(\Om)$. The formula proves independence of $h$ and of the positions of the distinct points in $P$. At $\rho=0$, $K_{0,H}=0$, giving $d_0=1$ directly.
\end{proof}

\begin{proof}[Proof of Corollary~\ref{cor:existence}]
Theorem~\ref{thm:compact} makes the degree in \eqref{eq:degreedef} admissible on a sufficiently large ball. If $d_\rho(\Om,H)\ne0$, the existence property of the Leray--Schauder degree gives a fixed point of $K_{\rho,H}$. By \eqref{eq:solutionmap}, any such fixed point is a weak solution of \eqref{maineq}.
\end{proof}

\begin{proof}[Proof of Corollary~\ref{cor:onemark}]
Set $c=\chi(\Om)$ and apply \eqref{eq:degree} with $m=1$. For $\varepsilon<\beta$, only the empty subset occurs if $k=0$, giving one. For $k\geq1$ the two terms give $\binom{k-c+1}k-\binom{k-c}{k-1}=\binom{k-c}k$. If $\varepsilon>\beta$, both floor values are $k$ and the terms cancel.
\end{proof}

\begin{proof}[Proof of Corollary~\ref{cor:equalmarks}]
For $r<1$, no ordinary point can occur. Each subset of at most $j$ points of $P$ contributes $(-1)^{|I|}$, and the alternating partial binomial sum gives the formula. For $m=2$ in the specified interval, the barycenter space consists of two points, so $1-\chi=-1$.
\end{proof}

\appendix
\section{Proofs of the technical lemmas}\label{app:technical}
The proofs follow the order in which the lemmas are used. We first establish the measure-data compactness needed for the nonlinear limit in Section~\ref{sec:compact}.
\begin{proof}[Proof of Lemma~\ref{lem:measure}]
Set $w_j=z_j-\varphi\geq0$ by comparison. For a fixed small $\delta>0$, \eqref{eq:BM} gives
\[
\sup_{j\in\N}\int_D e^{\delta w_j}\dd x<\infty.
\]
Thus $(w_j)_{j\in\N}$ is bounded in every finite $L^a$. Subtract the equation for $\varphi$ and test with $1-(1+w_j)^{-1}$. Inequality \eqref{eq:monotone} yields
\[
d_N\int_D\frac{|\nabla w_j|^N}{(1+w_j)^2}\dd x\leq M.
\]
For $1\leq a<N$, H\"older's inequality gives
\[
\int_D|\nabla w_j|^a\dd x\leq(M/d_N)^{a/N}
\Big(\int_D(1+w_j)^{2a/(N-a)}\dd x\Big)^{(N-a)/N}\leq C_a.
\]
Rellich's compactness theorem gives a subsequence with $w_j$ convergent in $L^1$ and almost everywhere. For $k>0$, set $T_k(t)=\max(-k,\min(t,k))$. The common trace makes $T_k(w_j-w_l)$ admissible in the difference equation, giving
\[
d_N\int_{\{x\in D:|w_j(x)-w_l(x)|<k\}}|\nabla w_j-\nabla w_l|^N\dd x\leq2Mk.
\]
Thus
\[
|\{|\nabla w_j-\nabla w_l|>\varepsilon\}|
\leq|\{|w_j-w_l|\geq k\}|+\frac{2Mk}{d_N\varepsilon^N}.
\]
First let $j,l\to\infty$, then $k\downarrow0$: the gradients are Cauchy in measure. For each $a<N$, use a bound at $a<a'<N$ to obtain uniform integrability and strong $L^a$ convergence. The common $\varphi\in W^{1,N}$ can be added back. Choosing $a>N-1$ gives convergence of $|\nabla z_j|^{N-2}\nabla z_j$ in $L^{a/(N-1)}$, hence in $L^1$, by continuity of $\xi\mapsto|\xi|^{N-2}\xi$ and uniform integrability. The distributional equation and trace pass to the limit. The proof uses only boundedness and energy regularity of the fixed $N$-harmonic function $\varphi$.
\end{proof}

With the analytic compactness established, we turn to the degree comparison of Section~\ref{sec:deformation}. The first step identifies the nonlinear solution map with its Galerkin approximations without changing the degree sign.
\begin{proof}[Proof of Lemma~\ref{lem:Galerkin}]
$A_n$ is the gradient of the strictly convex coercive restriction of $Q$, hence a homeomorphism of $E_n$. Set $K_n=A_n^{-1}\Gcal_n:E\to E_n$. On any bounded $D\subset E$, testing against $K_n(u)$ gives $\norm{K_n(u)}_E^{N-1}\leq\norm{\Gcal(u)}_{E^*}$, so these maps and $K_\rho$ have uniformly bounded images. Write $w=K_\rho(u)$ and $w_n=K_n(u)$. Galerkin orthogonality and \eqref{eq:monotone}, for $z\in E_n$, give
\[
d_N\norm{w_n-w}_E^N\leq\pair{A(w_n)-A(w)}{w_n-w}
=\pair{A(w_n)-A(w)}{z-w}\leq C_D\norm{z-w}_E.
\]
The relatively compact set $K_\rho(D)$ is approximated uniformly by the increasing dense spaces $E_n$. Thus $K_n\to K_\rho$ uniformly on bounded sets.

Compactness implies $\inf_{u\in\partial U}\norm{u-K_\rho(u)}_E>0$: a contradicting sequence would converge via its compact images to a boundary fixed point. For large $n$, the straight homotopy between $K_\rho,K_n$ is admissible. Leray--Schauder finite-dimensional reduction gives
\[
\degLS(\Id-K_\rho,U,0)=\degB(\Id-K_n,U\cap E_n,0).
\]
To fix the sign consider
\[
H_s(u)=(1-s)[A_n(u)-A_n(K_n(u))]+s[u-K_n(u)].
\]
For $d=u-K_n(u)\ne0$,
\[
(H_s(u),d)_n=(1-s)\pair{A(u)-A(K_n(u))}{d}+s\norm d_n^2>0.
\]
Thus the zeros are exactly $u=K_n(u)$ throughout the homotopy. The endpoints are $A_n-\Gcal_n$ and $\Id-K_n$, proving the degree identity with the stated sign.
\end{proof}

The Galerkin comparison will be applied to a bounded set between two energy levels. We next prove the deformation needed for this localization while keeping trajectories inside $\Mcap$.
\begin{proof}[Proof of Lemma~\ref{lem:capstrip}]
Every family critical point has $F<C-2$. By compactness and the hypothesis, choose $\varepsilon\in(0,1/4)$ so that $J$ has no critical point in
\[
\{a-2\varepsilon\leq J\leq b+2\varepsilon, F\leq C+2\varepsilon\}.
\]
Otherwise exact critical points tending to this excluded strip would converge to a critical point in it. Set $D=\{a-\varepsilon<J<b+\varepsilon,F<C+\varepsilon\}$. It is bounded by \eqref{eq:normidentity}. Lemma~\ref{lem:Splus} shows that
\[
\inf_{u\in D}\norm{J'(u)}_{E^*}>0,\quad 
\inf_{\substack{u\in D\\F(u)>C-1}}\min_{\lambda\in[\rho-\delta,\rho]}\norm{J_\lambda'(u)}_{E^*}>0,
\]
when the indicated sets are nonempty: a violating bounded sequence converges strongly to an excluded critical point of $J$ or a family critical point with $F\geq C-1$.

On $D\cap\{F<C-1/2\}$ choose local unit directions decreasing $J$. On $D\cap\{F>C-1\}$ separate the segment joining $J',F'$ from zero and choose directions decreasing both. These open sets cover $D$. Shrink neighborhoods in the first group to remain in $F<C-1/2$. A locally Lipschitz partition gives a field $V_D$ with $\norm{V_D}_E\leq1$, $J'V_D\leq-\eta<0$ on $D$, and $F'V_D\leq-\eta$ whenever $F\geq C-1/2$. The uniform lower bounds supply one $\eta>0$.

For $u_0\in\Mcap$ with $a<J(u_0)\leq b$, the flow stays inside $\Mcap$, since $F$ strictly decreases near its boundary, and stays below its initial $J$ value. Before $J=a$ the trajectory remains in $D$. Bounded speed gives continuation, and $\tau(u_0)\leq(J(u_0)-a)/\eta$ gives finite contact. Continuous dependence and strict crossing give continuity of $\tau$ on the moving part; the displayed estimate gives continuity when extended by zero on $\Mcap\cap J^a$. Stop at time $s\tau(u_0)$ and leave the lower set fixed. This is the required sublevel-preserving strong deformation retraction.
\end{proof}

We can now compare the constrained relative homology with a finite-dimensional section. The proof first establishes \eqref{eq:homologyreduction} and then uses the degree sign fixed in Lemma~\ref{lem:Galerkin} to obtain \eqref{eq:EulerBrouwer}.
\begin{proof}[Proof of Lemma~\ref{lem:penalty}]
Set $\sigma=\rho/\delta$, $\tau=\rho/\delta-1>0$. For a finite-dimensional $Y\subset E$, $L\geq0$, and $0\leq t\leq1$, define
\[
d_Y(u)=\dist(u,Y),\quad  J_{t,Y}=J+tLd_Y,\quad  F_{t,Y}=F+tLd_Y.
\]
Adding $tLd_Y$ to both functionals transforms \eqref{eq:normidentity} into
\begin{equation}\label{eq:penaltyidentity}
Q(u)+tLd_Y(u)=\sigma F_{t,Y}(u)-\tau J_{t,Y}(u).
\end{equation}
Thus $\{a-1\leq J_{t,Y}\leq b+1,F_{t,Y}\leq C+1\}$ is contained in one fixed ball $B_R^E$, $R>1$, independently of $Y,t,L$. Choose $L>2\max\{\Lip(J),\Lip(F)\}+1$, with the Lipschitz constants on $B_{3R+1}^E$. These constants are finite by Proposition~\ref{prop:functional}.

We first establish uniform separation near the three constraint boundaries. For a locally Lipschitz function $G$, let
\[
G^\circ(u;v)=\limsup_{z\to u, s\downarrow0}\frac{G(z+sv)-G(z)}s,
\quad
\partial_CG(u)=\{\xi\in E^*:\pair\xi v\leq G^\circ(u;v), \forall\, v\in E\}.
\]
For the convex function $d_Y$, $\partial_Cd_Y$ equals its convex subdifferential, and
\begin{equation}\label{eq:distancesubgrad}
\partial d_Y(u)=\{\xi\in E^*: \norm\xi\leq1, \xi|_Y=0, \pair\xi u=d_Y(u)\}.
\end{equation}
The inequality defining a convex subgradient proves the annihilation and norm conditions; evaluating it at $0$ and $2u$ gives the last equality. Conversely these three conditions imply $d_Y(v)\geq d_Y(u)+\pair\xi{v-u}$. Hahn--Banach on $E/Y$ gives nonemptiness. The $C^1$ sum rule is exact here:
\[
\partial_CJ_{t,Y}(u)=J'(u)+tL\partial d_Y(u),\quad
\partial_CF_{t,Y}(u)=F'(u)+tL\partial d_Y(u).
\]
Indeed the differentiable summand has a uniform first-order expansion in the definition of $G^\circ$, and the convex directional derivative of $d_Y$ is the support function of \eqref{eq:distancesubgrad}.

There are $\varepsilon,\eta>0$ such that for all sufficiently large $n$, $Y=E_n$, all $t\in[0,1]$, and all $u$ with
\[
a-\varepsilon\leq J_{t,Y}(u)\leq b+\varepsilon,\quad F_{t,Y}(u)\leq C+\varepsilon,
\]
lying within $\varepsilon$ of at least one boundary level $J_{t,Y}=a$, $J_{t,Y}=b$, $F_{t,Y}=C$, the convex hull of the two subdifferentials has distance at least $\eta$ from zero. If not, a sequence with $n\to\infty$, difference from the boundary level tending to zero, and bounded $u_n$ would satisfy
\begin{equation}\label{eq:penaltycrit}
J_{\lambda_n}'(u_n)+t_nL\xi_n\to 0\quad\text{in }E^*,\quad 
\lambda_n\in[\rho-\delta,\rho],\quad\xi_n\in\partial d_{E_n}(u_n).
\end{equation}
Here convexity of $\partial d_{E_n}$ combines the two subgradients into one; $L$ was fixed before $n$ was chosen. Extract $u_n\weak u$, $\lambda_n\to\lambda$, and choose $y_n\in E_n$ with $y_n\to u$. Equation \eqref{eq:distancesubgrad} gives
\[
\pair{\xi_n}{u_n-u}=d_{E_n}(u_n)+\pair{\xi_n}{y_n-u}\geq-o(1).
\]
Pairing \eqref{eq:penaltycrit} with $u_n-u$ gives
\[
\limsup_{n\to\infty}\pair{A(u_n)}{u_n-u}\leq0,
\]
because $B'(u_n)$ converges strongly and the distance-subgradient pairing is nonnegative up to $o(1)$. Subtracting $\pair{A(u)}{u_n-u}\to0$ and using \eqref{eq:monotone} gives $u_n\to u$. Therefore
\[
d_{E_n}(u_n)\leq\norm{u_n-u}+\norm{u-y_n}\to 0.
\]
For every fixed $v$ in the increasing dense union of the $E_n$, the term $\pair{\xi_n}v$ vanishes eventually. Testing \eqref{eq:penaltycrit} with $v$, and then using density, gives $J_\lambda'(u)=0$. The modified values converge to $J(u),F(u)$, since $t_nL d_{E_n}(u_n)\to0$. Passing to a subsequence, one of the three defining functions tends to its boundary value, so $J(u)=a$, $J(u)=b$, or $F(u)=C$. Each alternative contradicts \eqref{eq:facegap}. The contradiction with \eqref{eq:facegap} proves separation uniformly in $t$ for all sufficiently large $n$.

Here is the local persistence needed to turn separation into a field. Fix a sufficiently large $Y=E_n$. The sets $\partial d_Y(u)$ are weak-* compact subsets of the unit dual ball. Their graph is closed for norm convergence of $u$ and weak-* convergence of $\xi$: pass to the limit in
\[
d_Y(v)\geq d_Y(u_j)+\pair{\xi_j}{v-u_j}\quad(v\in E).
\]
For a fixed $v\in E$, the support function $u\mapsto\max_{\xi\in\partial d_Y(u)}\pair\xi v$ is upper semicontinuous. Otherwise maximizing subgradients along a contradicting sequence would have a weak-* convergent subsequence in the unit ball (which is metrizable because $E$ is separable), and the closed graph would contradict the alleged strict increase. The same argument, with $t_j\to t$ and $\lambda_j\to\lambda$, applies to
\[
\mathscr G(t,u)=\{J_\lambda'(u)+tL\xi:\lambda\in[\rho-\delta,\rho],\ \xi\in\partial d_Y(u)\}.
\]
The set $\mathscr G(t,u)$ is exactly the weak-* compact convex hull of $\partial_CJ_{t,Y}(u)$ and $\partial_CF_{t,Y}(u)$. Continuity of $J',F'$ and compactness of the parameter interval verify every extra limit. Separate $\mathscr G(t_0,u_0)$ from the open dual ball of radius $\eta/2$. Reflexivity identifies the separating functional with $v\in E$, $\norm v\leq1$, oriented so that
\[
\sup_{\zeta\in\mathscr G(t_0,u_0)}\pair\zeta v\leq-\eta/2.
\]
Upper semicontinuity makes the same constant vector satisfy a bound $-\eta/4$ on an open neighborhood of $(t_0,u_0)$. Thus pointwise separation supplies a locally valid common descent direction, with no assertion that the subgradients vary in norm.

For the continuation, suppress $Y$ and write $J_t,F_t$. Define
\[
N_t=\{a\leq J_t\leq b,F_t\leq C\},\quad L_t=\{J_t=a,F_t\leq C\},\quad Q_t=N_t/L_t.
\]
A quotient by an empty subset means adjoining a disjoint basepoint; use this convention throughout. The separation holds on an open neighborhood of the three constraint boundaries in $[0,1]\times E$. Choose smaller closed neighborhoods of these boundaries inside it. The local directions just proved, scaled by a common constant, and the zero vector on the complement of these smaller closed neighborhoods, form an open cover construction. A locally finite locally Lipschitz partition on the whole product gives a bounded locally Lipschitz field $W:[0,1]\times E\to E$ with
\begin{equation}\label{eq:Clarkedecay}
J_t^\circ(u;W(t,u))\leq-LR-1,\quad 
F_t^\circ(u;W(t,u))\leq-LR-1
\end{equation}
on the smaller boundary neighborhoods. Their closures can be chosen inside the bounded set defined by the enlarged level inequalities. Convex combinations preserve both inequalities; on these neighborhoods the zero part has zero coefficient. Constructing the partition on the whole product avoids extending a field across a boundary with uncontrolled local Lipschitz constants.

Let $\gamma:[0,T]\to[0,1]$ be piecewise $C^1$, $|\dot\gamma|\leq1$, from $s$ to $t$. Solve
\[
\dot z(r)=W(\gamma(r),z(r)),\quad  z(0)=u.
\]
The equation for $z$ has a unique local solution. If its maximal interval had a finite endpoint, boundedness of $W$ would make $z(r)$ Cauchy at that endpoint. Completeness of $E$ provides a limit, and local existence there continues the solution. Thus it exists throughout $[0,T]$. Continuous dependence on initial data and on a uniformly varying path $\gamma$ follows by covering each fixed compact solution curve by finitely many neighborhoods on which $W$ has a Lipschitz bound, and applying Gronwall's inequality on these neighborhoods in order. No bounded subset of $E$ is being treated as compact. The elementary Lipschitz chain inequality gives almost everywhere
\[
\frac{\dd}{\dd r}J_{\gamma(r)}(z(r))
\leq J_{\gamma(r)}^\circ(z(r);\dot z(r))+L\dot\gamma(r)d_Y(z(r)),
\]
and likewise for $F$. On the relevant constrained set $d_Y(z)\leq R$, so \eqref{eq:Clarkedecay} makes the defining function of every active moving constraint decrease at rate at least one. The chain inequality for $J_{\gamma(r)}(z(r))$ also follows directly by taking upper difference quotients and using local Lipschitz continuity to replace increments of the $C^1$ path by $h\dot z$; it requires no differentiability of $d_Y$. An orbit from $N_s$ therefore stays within the upper and cap constraints until its first contact with the moving lower level.

Let $e_\gamma(u)$ be the first contact time with the moving lower level, with value $+\infty$ if there is no contact during $[0,T]$, and value zero for $u\in L_s$. Define
\[
\Theta_\gamma(u)=
\left\{\begin{aligned}
&*,&&e_\gamma(u)\leq T,\\
&q_t(z(T)),&&e_\gamma(u)>T,
\end{aligned}\right.
\quad  u\in N_s,
\]
where $q_t:N_t\to Q_t$ is the quotient map and $*$ is its basepoint. The inequalities at the active upper level and constraint boundary prevent an orbit from leaving those constraints before lower contact. This follows from strict decrease in a neighborhood of each active constraint boundary; it remains valid at simultaneous contacts. Thus every nonexiting endpoint belongs to $N_t$.

We verify continuity before passing to the quotient. If $e_\gamma(u)>T$, the positive continuous gap $J_{\gamma(r)}(z(r))-a$ has a positive minimum on $[0,T]$, so nearby data and paths also do not exit, and flow continuity applies. If $e_\gamma(u)<T$, the gap is strictly negative just after first contact. Nearby trajectories therefore exit before $T$. If $e_\gamma(u)=T$, any sequence of nearby nonexiting endpoints converges in $N_t$ to $z(T)\in L_t$, and its quotient images converge to $*$. Exiting endpoints already have image $*$. These assertions also hold when the path varies in any of the continuous parameter families below with fixed endpoints. At $u\in L_s$, strict crossing gives exit for all nearby initial data when the duration is positive. When a constant-path duration decreases to zero, bounded speed makes each nonexiting endpoint tend to $u$, so its image again tends to $*$. Since these domains of representatives and parameters are metric, these sequential verifications prove continuity into the quotient, even if the quotient itself is not metrizable.

If $L_s\ne\varnothing$, $\Theta_\gamma$ is constant on $L_s$ and hence induces a pointed continuous map $Q_s\to Q_t$. Joint continuity of each induced homotopy follows from
\begin{equation}\label{eq:quotientproduct}
q_s\times\Id_I:N_s\times I\to  Q_s\times I
~\text{being a quotient map},\quad  I=[0,1].
\end{equation}
Since $I$ is locally compact Hausdorff, Whitehead's quotient-product lemma \cite[p.~288]{Michael} gives \eqref{eq:quotientproduct}. No local compactness of $N_s$, $L_s$, or $Q_s$ is required. If $L_s=\varnothing$, we have $Q_s=N_s\sqcup\{*\}$ and extend all maps and homotopies by fixing its isolated basepoint; no quotient-product argument is needed in that case. The same convention applies when $N_s=\varnothing$.

Use the paths $\gamma(r)=r$ and $\gamma(r)=1-r$ on $[0,1]$ to obtain pointed maps $Q_0\to Q_1$ and $Q_1\to Q_0$. Concatenation agrees with composition by uniqueness of the flow until first exit and the convention that an orbit stays at the basepoint after its first exit. To verify the homotopy inverse assertion explicitly, set
\[
\ell(r)=\left\{\begin{aligned}
  &r,&&0\leq r\leq1,\\&2-r,&&1\leq r\leq2.  
\end{aligned}\right.
\]
The family $(1-v)\ell(r)$, $0\leq v\leq1$, contracts the loop from zero to one and back to the constant path at zero, on the fixed time interval $[0,2]$. Its endpoints remain zero and its speed is at most one. The preceding joint continuity therefore gives a homotopy of self-maps of $Q_0$. The terminal map is the constant-parameter stopped flow for time two, which is homotopic to the identity by using durations $2(1-v)$. For $Q_1$ use the family $1-(1-v)\ell(r)$. This proves $Q_0\simeq Q_1$ as pointed spaces, including when an exit set is empty.

Before identifying relative homology, we establish the corresponding separation near the boundary levels within $Y$. For all sufficiently large $n$, the restricted covectors
\[
\{J_\lambda'(u)|_{E_n}:\lambda\in[\rho-\delta,\rho]\}
\]
stay a positive distance from zero near the three constraint boundaries of $\{u\in E_n:a\leq J(u)\leq b,F(u)\leq C\}$. The dual norms here are induced by $\norm{\cdot}_E$ on $E_n$. Otherwise there are $n\to\infty$, bounded $u_n\in E_n$, parameters $\lambda_n$ for which the restricted norms tend to zero and the difference from at least one boundary level tends to zero. Pass to $u_n\weak u$, $\lambda_n\to\lambda$, and choose $y_n\in E_n$ with $y_n\to u$. Testing against $u_n-y_n$ gives
\[
\pair{A(u_n)}{u_n-u}
=\pair{J_{\lambda_n}'(u_n)}{u_n-y_n}
 +\lambda_n\pair{B'(u_n)}{u_n-y_n}
 +\pair{A(u_n)}{y_n-u}\to 0.
\]
The first term tends to zero by the restricted norm bound, the second by compactness of $B'$, and the third by boundedness of $A(u_n)$ and strong convergence of $y_n$. Monotonicity implies $u_n\to u$. Tests on the dense union then give $J_\lambda'(u)=0$ on an excluded boundary level, contradicting \eqref{eq:facegap}. Choose $n$ large enough for both this separation and the separation for the modified functionals. Local separation and a partition of unity now give a field taking values in $Y$ and decreasing both restricted functions near the relevant constraint boundaries.

Set $P_t=\{J_t\leq b,F_t\leq C\}$ and $A_t=\{J_t\leq a,F_t\leq C\}$. When $A_t\ne\varnothing$, the pair $(P_t,A_t)$ is a good pair in the following precise sense: $A_t$ is closed and is a strong deformation retract of a neighborhood in $P_t$. Indeed take
\[
V_t=\{u\in P_t:J_t(u)<a+\varepsilon_1\}
\]
with $\varepsilon_1>0$ chosen so that $V_t\setminus A_t$ is contained in the neighborhood of the lower level. Flow at the fixed parameter $t$ and stop at $J_t=a$, leaving $A_t$ fixed. The cap is invariant, and the moving portion is bounded by \eqref{eq:penaltyidentity}. Uniform strict decrease gives a contact time bounded by a constant times $J_t(u)-a$. Continuous dependence and strict crossing give continuity of this time on $V_t\setminus A_t$, and the bound makes its extension by zero continuous on $A_t$. The stopped homotopy stays in $V_t$ and fixes $A_t$.

The same construction using the $Y$-valued field just obtained proves that
\[
P^Y=Y\cap\{J\leq b,F\leq C\},\quad 
A^Y=Y\cap\{J\leq a,F\leq C\}
\]
is a good pair whenever $A^Y\ne\varnothing$. Its moving part is contained in a compact closed collar, so it has the same uniform contact bound. This separate argument is necessary because the ambient field need not be tangent to $Y$.

For the pairs $(P_t,A_t)$ and $(P^Y,A^Y)$, \cite[Proposition 2.22]{Hatcher} gives
$H_*(P_t,A_t)\cong\widetilde H_*(P_t/A_t)$ and
$H_*(P^Y,A^Y)\cong\widetilde H_*(P^Y/A^Y)$.
For an empty lower member use the direct identities
$H_*(P,\varnothing)=H_*(P)=\widetilde H_*(P\sqcup\{*\})$ instead.
The closed cover $P_t=A_t\cup N_t$, with intersection $L_t$, gives
\[
P_t/A_t\cong N_t/L_t.
\]
One map is induced by inclusion of $N_t$ into $P_t$. For its inverse, use the constant basepoint map on $A_t$ and the quotient map on $N_t$. They agree on $L_t$, so the closed pasting lemma gives a continuous map on $P_t$, which descends to $P_t/A_t$. The two maps are inverse. If $L_t$ is empty but $A_t$ is nonempty, the two pieces of the closed cover are also open, and the collapsed $A_t$ is precisely the adjoined isolated basepoint. The all-empty case follows from the same convention. The identical argument applies to $(P^Y,A^Y)$. We obtain
\begin{equation}\label{eq:relativecontinuation}
H_*(P_0,A_0)\cong H_*(P_1,A_1).
\end{equation}

At $t=1$ let $\pi:E\to Y$ be the metric projection.
A nearest point exists because $Y$ is finite dimensional and is unique
because $E$, with its gradient norm, is strictly convex.
The projection is continuous: if $u_j\to u$, then
$\norm{\pi(u_j)}\leq2\norm{u_j}$, and every subsequential limit of
$(\pi(u_j))$ minimizes the distance to $u$. Uniqueness identifies that
limit with $\pi(u)$. Thus $\pi|_Y=\Id$ and
$\norm{u-\pi(u)}=d_Y(u)$. For bounded representatives $u\in N_1$, set $v_s=(1-s)u+s\pi(u)$. Then
\[
d_Y(v_s)=(1-s)d_Y(u),\quad  \norm{v_s}\leq2R.
\]
Let $K_0$ be the larger of the two Lipschitz constants used in choosing $L$. Since
$\norm{v_{s_2}-v_{s_1}}=(s_2-s_1)d_Y(u)$ for $s_2\geq s_1$, we have
\[
\begin{split}
J_{1,Y}(v_{s_2})-J_{1,Y}(v_{s_1})
&\leq(K_0-L)(s_2-s_1)d_Y(u)\\
&\leq-(s_2-s_1)d_Y(u),
\end{split}
\]
and the same for $F_{1,Y}$. The distance identity follows in the normed quotient $E/Y$ because $[v_s]=(1-s)[u]$. Thus the path preserves upper and cap constraints; once it reaches $A_1$ it stays there. For $u\in L_1$ the whole path is collapsed. Composition with $P_1\to P_1/A_1$ therefore yields a homotopy on $N_1/L_1$, jointly continuous by \eqref{eq:quotientproduct}, that retracts it to its subspace represented in $Y$. Only the bounded representatives are used; no Lipschitz estimate on the unbounded $A_1$ is needed. The finite-dimensional quotient has its expected topology. The set $Y\cap N_1$ is compact. The space $N_1/L_1$ is Hausdorff because $N_1$ is metric and $L_1$ is closed: two points outside $L_1$ have disjoint small neighborhoods avoiding it; a point outside $L_1$ and the collapsed set are separated by a small ball and the complement of a larger closed ball, both saturated for the quotient. If $L_1=\varnothing$, this is immediate. Thus the natural injection
\[
(Y\cap N_1)/(Y\cap L_1)\to  N_1/L_1,
\]
with the stated basepoint convention, is a homeomorphism onto its image. The homotopy fixes that image because $\pi|_Y=\Id$, and its terminal map has image there. The preceding good-pair and closed-cover identifications, applied within $Y$, now give
\[
H_*(P_1,A_1)\cong
H_*(Y\cap\{J\leq b,F\leq C\},Y\cap\{J\leq a,F\leq C\}).
\]
Together with \eqref{eq:relativecontinuation} this proves \eqref{eq:homologyreduction}.

Having proved \eqref{eq:homologyreduction}, we turn to \eqref{eq:EulerBrouwer}. Set $f=J|_Y$, $g=F|_Y$. The separation of the restricted derivatives near the boundary levels applies to $f,g$ on this compact constrained set.

Approximate $f,g$ in $C^1$ on a neighborhood of this set by functions $\widehat f,\widehat g$ smooth there, keeping their value errors at most $\varepsilon$ on all of $Y$ and the perturbations supported in a larger ball. Local convolution followed by a cutoff gives such approximations. For $f_s=(1-s)f+s\widehat f$, $g_s=(1-s)g+s\widehat g$, the identity $Q=\sigma g-\tau f$ implies
\[
Q(y)\leq\sigma C-\tau a+(\sigma+\tau)\varepsilon
\quad\text{ on }\,\{a\leq f_s\leq b,g_s\leq C\}.
\]
All sets defined by the perturbed inequalities are contained in a fixed ball, and their boundaries stay in the chosen neighborhoods for small $\varepsilon$, and small $C^1$ errors preserve separation. Also $|\partial_sf_s|,|\partial_sg_s|\leq\varepsilon$. A fixed scaling of the common field dominates these parameter derivatives, so the same stopped quotient continuation identifies the original and perturbed relative homology.

Choose $\widehat f$ Morse on this ball by a small generic linear perturbation of a smooth approximation: a regular value of its gradient gives nonsingular Hessians at all resulting critical points. Sard's theorem gives such choices arbitrarily small. The perturbation is cut off outside the larger ball. Further arbitrarily small perturbations of constants and of $\widehat g$ make the boundary levels and their intersections regular and transverse, without destroying separation. All resulting critical points in the constrained set are finite, since they are isolated in a compact subset away from its boundary. Their critical values can be made distinct by adding separate small constants in disjoint neighborhoods of these points, with cutoffs supported in slightly larger neighborhoods. The Hessians and critical points in the smaller neighborhoods are unchanged, and a sufficiently small derivative perturbation introduces no zero on their compact complement.

The constrained calculation is on $\widehat{\Mcap}=\{\widehat g\leq C\}$. At a boundary point between the two energy levels there exists a vector $v$ with $d\widehat g(v)<0$ and $d\widehat f(v)<0$. Critical points of $\widehat f|_{\{\widehat g=C\}}$ may exist, but there
\[
\nabla\widehat f=\lambda\nabla\widehat g,\quad  \lambda>0,
\]
as pairing with $v$ shows. At such a tangential critical point, $d\widehat g(-\nabla\widehat f)=-\lambda|\nabla\widehat g|^2<0$, so the negative gradient points strictly into the constraint. At the other boundary points use the common descending direction supplied by separation. Choose disjoint Morse neighborhoods of the interior critical points, with closures in $\{\widehat g<C\}$. On their complement take local fields decreasing $\widehat f$, also decreasing $\widehat g$ near the constraint boundary. Patch by a partition of unity, keeping the usual Morse descending field in smaller Morse neighborhoods. Compactness gives a uniform descent rate on every closed strip outside these neighborhoods, and the inward cap condition prevents boundary exit. Thus all such portions retract to their lower levels in finite time.

For a critical point of index $k$, \cite[Lemma 2.2]{MilnorMorse} supplies coordinates $(x_-,x_+)\in\R^k\times\R^{\dim Y-k}$ with $\widehat f=\widehat f(y)-|x_-|^2+|x_+|^2$. A sufficiently small Morse neighborhood is wholly inside the cap. Across a small interval containing its critical value, the local attachment is the pair $(D^k\times D^{\dim Y-k},\partial D^k\times D^{\dim Y-k})$, which retracts relatively to $(D^k,\partial D^k)$ by shrinking the positive coordinates. The cell attachment theorem \cite[Theorem 3.2]{MilnorMorse} identifies this local crossing; the descending field on its complement gives its attachment to the constrained lower sublevel. The compactness hypothesis of that theorem is supplied here by the bounded finite-dimensional set between the two levels, and its local construction is inside the cap. Distinct critical values allow these crossings to be ordered. All remaining strips retract by the preceding field, including at the cap boundary. Hence each interior critical point contributes one cell of its Morse index, and the boundary contributes none.

After collapse of the lower member the quotient thus has the homotopy type of a finite CW complex, with its basepoint and one cell of dimension $\operatorname{ind}(y)$ for each interior critical point. This concerns the relative quotient, not a finite CW model for the unbounded lower set. The relative Euler characteristic, equivalently the reduced Euler characteristic of this pointed quotient, is
\[
\sum_{\substack{y\in\widehat U\\\nabla\widehat f(y)=0}}(-1)^{\operatorname{ind}(y)}
=\degB(\nabla\widehat f,\widehat U,0),\quad 
\widehat U=\{a<\widehat f<b,\widehat g<C\}.
\]
There are no gradient zeros in the neighborhoods of the boundary levels throughout the perturbation. Excision gives one common isolating open set inside the sets defined by the original and perturbed inequalities. Brouwer homotopy invariance gives \eqref{eq:EulerBrouwer}. The cell count gives finite total rank, and \eqref{eq:homologyreduction} transfers it to the original pair, without any nondegeneracy assumption on the original functional.
\end{proof}

This completes the finite-dimensional comparison. For the local contractions in Section~\ref{sec:topology}, it remains to prove the averaging and dilation estimates, beginning with the construction of constant spherical traces.
\begin{proof}[Proof of Lemma~\ref{lem:radialize}]
Let $\mathcal R_cu$ be the spherical average about $c$. Choose a smooth radial cutoff $\zeta$, with $0\leq\zeta\leq1$, vanishing on the inner ball and near the outer sphere, equal to one on an intermediate annulus, with $|s\zeta'(s)|\leq\eta$. The two transitions can be made in the variable $\ln  s$; the annular ratio can be chosen large enough for any prescribed $\eta>0$. Set $\mathcal A_tu=u+t\zeta(\mathcal R_cu-u)$ in the ball and $u$ outside. Then
\[
\nabla\mathcal A_tu=(1-t\zeta)\nabla u+t\zeta\nabla\mathcal R_cu
+t\zeta'(\mathcal R_cu-u)\frac{x-c}{|x-c|}.
\]
On each sphere, convexity and Jensen's inequality bound the integral of the $N$th power of the first two terms by that of $|\nabla u|^N$. The Poincar\'e inequality on spheres bounds the last term's $N$th power by $C_N\eta^N$ times the same integral. Using $|v+w|^N\leq(1+\delta)|v|^N+C_{N,\delta}|w|^N$, first choose $\delta<\varepsilon$ and then $\eta$ small enough. Radial integration proves the energy bound.

Jensen's inequality for the exponential function on each sphere gives
\[
\int_{\partial B_s(c)}e^{\mathcal A_tu}\dd S_x\leq\int_{\partial B_s(c)}e^u\dd S_x
\]
for almost every $s$. Multiply by $s^{N\widehat\alpha}$, integrate, and use \eqref{eq:ballweight}; the weighted mass in the ball grows by at most $\kappa^{-1}$. The exterior is fixed. The lower mass bound comes from the fixed inner ball, and the functional bound follows. Spherical averaging and cutoff multiplication are bounded linear Sobolev operators on the annulus; the same estimates and the Poincar\'e inequality control the global $E$ norm. The cutoff gives equal traces at the joining spheres, and dependence on $t$ is continuous. For disjoint balls the energies add and the preserved inner masses are summed, proving all assertions.
\end{proof}

The constant spherical traces obtained in Lemma~\ref{lem:radialize} allow the following dilation and logarithmic interpolation.
\begin{proof}[Proof of Lemma~\ref{lem:dilate}]
At the inner joining sphere $\gamma_N\theta_i\ln  s_i=\gamma_N\ln  s$, and at the outer sphere the logarithmic term is zero. The traces agree, so Sobolev gluing applies. Critical-dimensional dilation preserves the core's $N$-energy. The energy on the annulus is
\[
\frac{S_N}{N}(\gamma_N\theta_i)^N\ln  s_i=qc_N\omega_i\ln  s.
\]
Summing over $1\leq i\leq k$ proves the energy identity in \eqref{eq:dilation}; replacing $w$ by $(1-t)w$ gives the identity for $\mathcal D_{s,t}w$. Write $h_i=h_{c_i}$. Core change of variables gives
\begin{align*}
\int_{B_{R_i/s_i}(c_i)}He^{\mathcal C_sw}\dd x
&=s_i^{-N(1+\widehat\alpha_i)}s^{\gamma_N}
\int_{B_i}h_i(c_i+(y-c_i)/s_i)|y-c_i|^{N\widehat\alpha_i}e^{w(y)}\dd y\\
&\geq\kappa s^q\int_{B_i}He^w\dd x,
\end{align*}
because $s_i^{N(1+\widehat\alpha_i)}=s^N$ and $\gamma_N-N=q$. For continuity, decompose $w=b_i+v_i$ in $B_i$, with $v_i\in W^{1,N}_0(B_i)$. Its zero extension has strongly continuous dilations, proved first for smooth functions and then by density. The truncated logarithmic terms depend continuously in $W^{1,N}$ on their positive truncation radii, including $s=1$. The constants $b_i$ are continuous linear trace averages on the specified subspace. These observations give joint continuity of both families.
\end{proof}

\medskip

\noindent\textbf{Data availability statement.}~	 Data sharing is not applicable to this article as no datasets were generated or analysed during the current study.

\medskip

\noindent\textbf{Declarations.}~ \textbf{Conflict of interest.} ~ On behalf of all authors, the corresponding author states that there is no conflict of interest.

\medskip

\noindent\textbf{AI assistance statement.} ~
During the preparation of this work the authors used ChatGPT (OpenAI) to assist with exposition and organization, to check intermediate arguments and references, and to identify inconsistencies in the manuscript. After using this tool, the authors reviewed and edited the content as needed and take full responsibility for the content of the published article.

\bigskip
\noindent\textsc{Mohameden Ahmedou}\par
\noindent Mathematisches Institut, Justus-Liebig-Universit\"at Giessen, Arndtstrasse 2, D-35392 Giessen, Germany\par
\noindent\href{mailto:Mohameden.Ahmedou@math.uni-giessen.de}{\texttt{Mohameden.Ahmedou@math.uni-giessen.de}}

\medskip
\noindent\textsc{Zhongwei Tang}\par
\noindent School of Mathematical Sciences, Laboratory of Mathematics and Complex Systems, Ministry of Education, Beijing Normal University, Beijing 100875, China\par
\noindent\href{mailto:tangzw@bnu.edu.cn}{\texttt{tangzw@bnu.edu.cn}}

\medskip
\noindent\textsc{Heming Wang}\par
\noindent State Key Laboratory of Mathematical Sciences, Academy of Mathematics and Systems Science, Chinese Academy of Sciences, Beijing 100190, China\par
\noindent\href{mailto:hmwang@amss.ac.cn}{\texttt{hmwang@amss.ac.cn}}

\begin{thebibliography}{99}

\bibitem{ABH26}M. Ahmedou, T. Bartsch, and Z. Hu,
Blow-up solutions for the steady state of the Keller--Segel
system on Riemann surfaces,
\emph{Calc. Var. Partial Differential Equations}
\textbf{65} (2026), no.~5, Paper No.~149, 34 pp.

\bibitem{AB2017} M. Ahmedou and M. Ben Ayed, Theory of ``critical points at infinity'' and a resonant singular Liouville-type equation, \emph{Adv. Nonlinear Stud.} \textbf{17} (2017), no.~1, 139--166.

\bibitem{ABL2017} M. Ahmedou, M. Ben Ayed, and M. Lucia, On a resonant mean field type equation: a ``critical point at infinity'' approach, \emph{Discrete Contin. Dyn. Syst.} \textbf{37} (2017), no.~4, 1789--1818.

\bibitem{AHW2025-1} M. Ahmedou, Z. Hu, and H. Wang, A degree-counting formula for a Keller--Segel equation on a surface with boundary, \emph{Commun. Contemp. Math.} (2026), Article 2650039, \href{https://doi.org/10.1142/S0219199726500392}{doi:10.1142/S0219199726500392}. Preprint version: \href{https://arxiv.org/abs/2506.12783v2}{arXiv:2506.12783v2}.

\bibitem{Bahri}A. Bahri, \emph{Critical Points at Infinity in Some Variational Problems}, Pitman Research Notes in Mathematics Series, vol.~182, Longman Scientific \& Technical, Harlow, 1989.

\bibitem{Battaglia19}L. Battaglia, A general existence result for stationary solutions to
the Keller--Segel system, \emph{Discrete Contin. Dyn. Syst.}
\textbf{39} (2019), no.~2, 905--926.

\bibitem{CLMP92}E. Caglioti, P.-L. Lions, C. Marchioro, and M. Pulvirenti, A special class of stationary flows for two-dimensional Euler equations: A statistical mechanics description, \emph{Comm. Math. Phys.} \textbf{143} (1992), no.~3, 501--525.

\bibitem{CLMP95}E. Caglioti, P.-L. Lions, C. Marchioro, and M. Pulvirenti, A special class of stationary flows for two-dimensional Euler equations: A statistical mechanics description. Part II, \emph{Comm. Math. Phys.} \textbf{174} (1995), no.~2, 229--260.

\bibitem{CM}A. Carlotto and A. Malchiodi, Weighted barycentric sets and singular Liouville equations on compact surfaces, \emph{J. Funct. Anal.} \textbf{262} (2012), no.~2, 409--450.

\bibitem{CGY93}S.-Y. A. Chang, M. J. Gursky, and P. C. Yang, The scalar curvature equation on 2- and 3-spheres, \emph{Calc. Var. Partial Differential Equations} \textbf{1} (1993), no.~2, 205--229.

\bibitem{CY87}S.-Y. A. Chang and P. C. Yang, Prescribing Gaussian curvature on $S^2$, \emph{Acta Math.} \textbf{159} (1987), nos.~3--4, 215--259.

\bibitem{CK94}S. Chanillo and M. Kiessling, Rotational symmetry of solutions of some nonlinear problems in statistical mechanics and in geometry, \emph{Comm. Math. Phys.} \textbf{160} (1994), no.~2, 217--238.

\bibitem{CL2002} C.-C. Chen and C.-S. Lin, Sharp estimates for solutions of multi-bubbles in compact Riemann surfaces, \emph{Comm. Pure Appl. Math.} \textbf{55} (2002), no.~6, 728--771.

\bibitem{CL03}C.-C. Chen and C.-S. Lin, Topological degree for a mean field equation on Riemann surfaces, \emph{Comm. Pure Appl. Math.} \textbf{56} (2003), 1667--1727.

\bibitem{CL15}C.-C. Chen and C.-S. Lin, Mean field equation of Liouville type with singular data: Topological degree, 
\emph{Comm. Pure Appl. Math.} \textbf{68} (2015), no.~6, 887--947.

\bibitem{CLW04}C.-C. Chen, C.-S. Lin, and G. Wang, Concentration phenomena of two-vortex solutions in a Chern--Simons model, \emph{Ann. Sc. Norm. Super. Pisa Cl. Sci.} (5) \textbf{3} (2004), no.~2, 367--397.

\bibitem{CEL}G. Ciraolo, P. Esposito, and X. Li, On the classification of solutions to a class of $N$-Liouville equations in $\R^N$, preprint (2026), \href{https://arxiv.org/abs/2604.10050}{arXiv:2604.10050}.

\bibitem{Coste}M. Coste, \emph{Real Algebraic Sets}, lecture notes, March 23, 2005, Sections~1.3--1.4, 1.8, and 3.1.

\bibitem{DPEFM15}M. del Pino, P. Esposito, P. Figueroa, and M. Musso, Nontopological condensates for the self-dual Chern--Simons--Higgs model, \emph{Comm. Pure Appl. Math.} \textbf{68} (2015), no.~7, 1191--1283.

\bibitem{DPEM10}M. del Pino, P. Esposito, and M. Musso, Two-dimensional Euler flows with concentrated vorticities, \emph{Trans. Amer. Math. Soc.} \textbf{362} (2010), no.~12, 6381--6395.

\bibitem{DJLW98}W. Ding, J. Jost, J. Li, and G. Wang, An analysis of the two-vortex case in the Chern--Simons Higgs model, \emph{Calc. Var. Partial Differential Equations} \textbf{7} (1998), no.~1, 87--97.

\bibitem{DHM}G. Dolzmann, N. Hungerb\"uhler, and S. M\"uller, Uniqueness and maximal regularity for nonlinear elliptic systems of $n$-Laplace type with measure valued right hand side, \emph{J. Reine Angew. Math.} \textbf{520} (2000), 1--35.

\bibitem{Dugundji}J. Dugundji, An extension of Tietze's theorem, \emph{Pacific J. Math.} \textbf{1} (1951), 353--367.

\bibitem{Esposito}P. Esposito, Isolated singularities for the $n$-Liouville equation, \emph{Calc. Var. Partial Differential Equations} \textbf{60} (2021), no.~4, Article 137.

\bibitem{EF14}P. Esposito and P. Figueroa, Singular mean field equations on compact Riemann surfaces, \emph{Nonlinear Anal.} \textbf{111} (2014), 33--65.

\bibitem{EM}P. Esposito and F. Morlando, On a quasilinear mean field equation with an exponential nonlinearity, \emph{J. Math. Pures Appl.} \textbf{104} (2015), 354--382.

\bibitem{Hanner}O. Hanner, Some theorems on absolute neighborhood retracts, \emph{Ark. Mat.} \textbf{1} (1951), 389--408.

\bibitem{Hatcher}A. Hatcher, \emph{Algebraic Topology}, Cambridge University Press, Cambridge, 2002.

\bibitem{HL}J. Hor\'ak and M. Lucia, A minimax theorem in the presence of unbounded Palais--Smale sequences, \emph{Israel J. Math.} \textbf{172} (2009), 125--143.

\bibitem{JWY18}A. Jevnikar, J. Wei, and W. Yang,
On the topological degree of the mean field equation with two parameters,
\emph{Indiana Univ. Math. J.}
\textbf{67} (2018), no.~1, 29--88.

\bibitem{Kallel}S. Kallel, Formal barycenter spaces with weights: the Euler characteristic, \emph{Topol. Methods Nonlinear Anal.} \textbf{53} (2019), no.~2, 801--823.

\bibitem{KW74}J. L. Kazdan and F. W. Warner, Curvature functions for compact 2-manifolds, \emph{Ann. of Math.} (2) \textbf{99} (1974), no.~1, 14--47.

\bibitem{KV}S. Kichenassamy and L. V\'eron, Singular solutions of the $p$-Laplace equation, \emph{Math. Ann.} \textbf{275} (1986), 599--615.

\bibitem{Kiessling93}M. K.-H. Kiessling, Statistical mechanics of classical particles with logarithmic interactions, \emph{Comm. Pure Appl. Math.} \textbf{46} (1993), no.~1, 27--56.

\bibitem{L1999} Y. Y. Li, Harnack type inequality: the method of moving planes, \emph{Comm. Math. Phys.} \textbf{200} (1999), no.~2, 421--444.

\bibitem{LW10}C.-S. Lin and C.-L. Wang, Elliptic functions, Green functions and the mean field equations on tori, \emph{Ann. of Math.} (2) \textbf{172} (2010), no.~2, 911--954.

\bibitem{LY13}C.-S. Lin and S. Yan, Existence of bubbling solutions for Chern--Simons model on a torus, \emph{Arch. Ration. Mech. Anal.} \textbf{207} (2013), no.~2, 353--392.

\bibitem{Lucia05}M. Lucia, A deformation lemma for a Moser--Trudinger type functional, \emph{Nonlinear Anal.} \textbf{63} (2005), no.~2, 282--299.

\bibitem{Lucia07}M. Lucia, A deformation lemma with an application to a mean field equation, \emph{Topol. Methods Nonlinear Anal.} \textbf{30} (2007), no.~1, 113--138.

\bibitem{Malchiodi}A. Malchiodi, Morse theory and a scalar field equation on compact surfaces, \emph{Adv. Differential Equations} \textbf{13} (2008), nos.~11--12, 1109--1129.

\bibitem{Malchiodi26}A. Malchiodi, On a degree-counting formula for elliptic equations
of exponential type, \emph{Discrete Contin. Dyn. Syst.}
\textbf{55} (2026), 288--305.

\bibitem{MR11}A. Malchiodi and D. Ruiz, New improved Moser--Trudinger inequalities and singular Liouville equations on compact surfaces, \emph{Geom. Funct. Anal.} \textbf{21} (2011), no.~5, 1196--1217.

\bibitem{Michael}E. Michael, Bi-quotient maps and cartesian products of quotient maps, \emph{Ann. Inst. Fourier (Grenoble)} \textbf{18} (1968), no.~2, 287--302.

\bibitem{MilnorCW}J. Milnor, On spaces having the homotopy type of a CW-complex, \emph{Trans. Amer. Math. Soc.} \textbf{90} (1959), 272--280.

\bibitem{MilnorMorse}J. Milnor, \emph{Morse Theory}, Annals of Mathematics Studies, vol.~51, Princeton University Press, Princeton, NJ, 1963.

\bibitem{Moser}J. Moser, A sharp form of an inequality by N. Trudinger, \emph{Indiana Univ. Math. J.} \textbf{20} (1971), 1077--1092.

\bibitem{NT99}M. Nolasco and G. Tarantello, Double vortex condensates in the Chern--Simons--Higgs theory, \emph{Calc. Var. Partial Differential Equations} \textbf{9} (1999), no.~1, 31--94.

\bibitem{Serrin}J. Serrin, Local behavior of solutions of quasi-linear equations, \emph{Acta Math.} \textbf{111} (1964), 247--302.

\bibitem{Stuart67}J. T. Stuart, On finite amplitude oscillations in laminar mixing layers, \emph{J. Fluid Mech.} \textbf{29} (1967), no.~3, 417--440.
\end{thebibliography}
\end{document}